\documentclass[10pt]{article}

\usepackage{amssymb,amsthm,amsmath,hyperref,txfonts}

\usepackage{color}

\usepackage{graphicx,float}

\numberwithin{equation}{section}

\newtheorem{thm}{Theorem}[section]
\newtheorem{lem}{Lemma}[section]
\newtheorem{cor}{Corollary}[section]
\newtheorem{prop}{Proposition}[section]
\newtheorem{rem}{Remark}[section]

\newcommand{\beq}{\begin{eqnarray}}
\newcommand{\eeq}{\end{eqnarray}}
\newcommand{\beqno}{\begin{eqnarray*}}
\newcommand{\eeqno}{\end{eqnarray*}}
\newcommand{\be}{\begin{equation}}
\newcommand{\ee}{\end{equation}}
\theoremstyle{definition}

\newcommand{\crl}{\nabla^{\perp}}

\newcommand{\otv}{\operatorname{div}}

\newcommand{\lnorm}[2]{\left\| #1 \right\|_{L^{#2}}}

\usepackage{authblk}
\hypersetup{hidelinks}
\usepackage[utf8]{inputenc}  
\usepackage[T1]{fontenc}
\usepackage{lmodern}         
\input{glyphtounicode}

\allowdisplaybreaks

\begin{document}
\title{\bf Large-bulk-viscosity limit for two-dimensional compressible micropolar fluids with vacuum}

\author{Zijian Du, Sen Yang, Xinglong You}

\date{}
\maketitle
 
\begin{abstract}

We consider the Cauchy problem for the two-dimensional compressible
micropolar fluid system on the torus $\mathbb T^2$. The initial density
is bounded and nonnegative and may contain vacuum, while the initial
velocity and microrotation belong to $H^1(\mathbb T^2)$. For sufficiently
large bulk viscosity $\nu$, we prove the global-in-time existence of
solutions without a smallness assumption on the $H^1$ size of the data,
provided that the initial compressible component is well prepared in the
sense that $\sqrt{\nu}\|\operatorname{div}u_0\|_{L^2}$ is uniformly
bounded. The proof combines the effective viscous flux with a coupled
energy scheme that exploits the vorticity--microrotation dissipation at
both the basic and material-derivative levels. For the linear pressure
law, we also prove uniqueness in the stated solution class by means of a
coupled backward parabolic problem, logarithmic interpolation, and
Osgood's lemma. For fixed initial data with $\operatorname{div}u_0=0$, we also justify the limit $\nu\to\infty$:
every sequence of the corresponding solutions admits a subsequence
converging to a solution of the two-dimensional inhomogeneous
incompressible micropolar system.

\vspace{4mm}



\end{abstract}

\section{Introduction}

\subsection{Background and motivation}

Micropolar fluid theory, introduced by Eringen \cite{Eringen1966},
extends the classical continuum description by assigning an independent
microrotational velocity to the fluid particles and allowing the Cauchy
stress tensor to be nonsymmetric; see also the monograph \cite{Lu1999}.
The model is intended to describe particle-spin effects that are absent
from the Newtonian theory and arise, for example, in drag problems for
micropolar media \cite{Hoffmann2007}. In two space dimensions the
microrotation is scalar, but it remains dynamically coupled to the
transverse part of the velocity. In this paper, we consider the Cauchy
problem on the normalized torus $\mathbb T^2$:
\begin{equation}\label{0}
\begin{cases}
\rho_t+\operatorname{div}(\rho u)=0,\\
(\rho u)_t+\operatorname{div}(\rho u\otimes u)+\nabla P(\rho)
=(\mu+\xi)\Delta u+(\mu+\lambda-\xi)\nabla\operatorname{div}u
-2\xi\nabla^\perp w,\\
(\rho w)_t+\operatorname{div}(\rho u w)+4\xi w
=\mu_2\Delta w+2\xi\nabla^\perp\!\cdot u,
\end{cases}
\qquad (t,x)\in\mathbb R_+\times\mathbb T^2.
\end{equation}
Here $\rho=\rho(t,x)\geq0$, $u=(u_1,u_2)=u(t,x)$, and
$w=w(t,x)$ denote the density, velocity, and microrotation,
respectively. We use
\[
\nabla^\perp w=(-\partial_2w,\partial_1w),\qquad
\nabla^\perp\!\cdot u=\partial_1u_2-\partial_2u_1,
\qquad \nu:=2\mu+\lambda,
\]
and assume throughout that $\mu,\mu_2,\xi,$ and $\nu$ are positive.
The pressure obeys the isentropic law $P(\rho)=\rho^\gamma$ with
$\gamma\geq1$. The system is supplemented with
\begin{equation}\label{eq:IC}
(\rho,u,w)|_{t=0}=(\rho_0,u_0,w_0).
\end{equation}

For sufficiently regular solutions, the total mass and linear momentum
are conserved, while the microrotational momentum satisfies a balance
law with damping:
\begin{equation}\label{con}
\begin{aligned}
&\int_{\mathbb T^2}\rho(t)\,dx=\int_{\mathbb T^2}\rho_0\,dx,\\
&\int_{\mathbb T^2}\rho u(t)\,dx
=\int_{\mathbb T^2}\rho_0u_0\,dx,\\
&\int_{\mathbb T^2}\rho w(t)\,dx
+4\xi\int_0^t\!\int_{\mathbb T^2}w\,dx\,d\tau
=\int_{\mathbb T^2}\rho_0w_0\,dx.
\end{aligned}
\end{equation}
We impose the normalization
\begin{equation}\label{norm}
\int_{\mathbb T^2}\rho_0\,dx=1,
\qquad
\int_{\mathbb T^2}\rho_0u_0\,dx=0.
\end{equation}

The regime studied here is distinguished by a large bulk viscosity
$\nu$. Indeed, the viscous part of the momentum equation can be written
as
\[
(\mu+\xi)(\Delta u-\nabla\operatorname{div}u)
+\nu\nabla\operatorname{div}u.
\]
Thus, increasing $\nu$ directly penalizes only the potential part of
the velocity. It neither gives enhanced dissipation for the solenoidal
part nor removes the coupling of that part with $w$. More precisely,
the terms $-2\xi\nabla^\perp w$ and
$2\xi\nabla^\perp\!\cdot u$ exchange energy between the
microrotation and the vorticity. This interaction persists in the limit
$\nu\to\infty$. Moreover, because the initial density may be nonconstant
and may vanish, the limiting density is transported by a divergence-free
velocity rather than being forced to a positive constant. The natural
limit is therefore an inhomogeneous incompressible micropolar system.

We next recall the results most closely related to the present work.
For incompressible micropolar equations, global strong solutions in
bounded or three-dimensional domains were studied in
\cite{Boldrini2010,Nowakowski2013,Yamaguchi}, and the long-time dynamics
were considered in \cite{Chen2006}. Well-posedness in critical spaces
was established in \cite{Chen2012}, while nonhomogeneous asymmetric
fluids were treated in \cite{Boldrini2003}. Two-dimensional models with
partial or fractional dissipation were investigated in
\cite{Dong2010,Dong2017,Ye2026}; see also \cite{Chamorro2025} for a
partial-regularity result. The long-time alignment of microrotation and
vorticity obtained in \cite{guterres2025} further reflects the
importance of the transverse coupling appearing in \eqref{0}.

For compressible micropolar models, a systematic one-dimensional global
theory was developed in \cite{Mujakovi1,Mujakovi2,Mujakovi3}.
Amirat--Hamdache \cite{Amirat2009} constructed finite-energy weak
solutions for a related compressible magnetic-fluid system containing
an angular-momentum equation. In several space dimensions, global weak
solutions with discontinuous density and vacuum were obtained in
\cite{Chen2015}, and global classical solutions with far-field vacuum
were considered in \cite{Zhang2025}. Heat-conducting compressible
micropolar fluids with far-field vacuum were studied in \cite{liu2025},
whereas \cite{wei2026} addressed a two-dimensional non-isentropic
initial-boundary value problem and derived a blow-up criterion. These
works provide essential existence and regularity results, but they do
not consider a large-bulk-viscosity limit for large data with vacuum.

When $\xi=0$, the density--velocity subsystem of \eqref{0} reduces to
the isentropic compressible Navier--Stokes equations. Their solution
theory includes global classical solutions near a non-vacuum equilibrium
\cite{Matsumura1980}, solutions with discontinuous density
\cite{Hoff1995,Hoff19952}, local strong or classical solutions with
nonnegative density \cite{Cho2004,Cho20042}, and global finite-energy
weak solutions \cite{Feireisl2001,Lions1998}. Global classical
solutions with small total energy but possibly large oscillations and
vacuum were established in \cite{Huang2012}. Of particular relevance
here is the large-bulk-viscosity approach of Danchin--Mucha. They proved
global results for large solutions and justified an incompressible limit
in \cite{R.Danchin1}; their later work \cite{R.Danchin2} allowed bounded
density with discontinuities and vacuum for the linear pressure law.
The corresponding uniform time-weighted estimates were further refined
by Liu--Zhang \cite{liu2026}.

The incompressible limit for compressible micropolar fluids was recently
studied by Han--Wu \cite{Han2025} in a perturbative critical-space
setting without vacuum, under an assumption on the global strong
solution of the limiting system. More recently, Lin--Liu--Zhang \cite{LLZ2026} studied the
combined limits of large bulk viscosity and vanishing angular
viscosity for the two-dimensional compressible micropolar
system on $\mathbb R^2$. Under a positive lower bound for
the density and suitable Sobolev regularity and
well-preparedness assumptions on the initial data, they
established global strong solutions and estimates uniform
in both viscosity parameters.  On the incompressible side,
Danchin--Mucha \cite{R.Danchin5} developed a vacuum theory for the
inhomogeneous Navier--Stokes equations, and
Guo--Qian--Zhang--Zheng \cite{guo2026} established global solutions and
time-weighted regularity for the two-dimensional inhomogeneous
incompressible micropolar system with bounded density. These results do
not supply the estimates, uniform in the bulk viscosity, that are needed
to pass from \eqref{0} with vacuum and large $H^1$ data to the
inhomogeneous incompressible micropolar dynamics.

The purpose of this paper is twofold. First, for every $\gamma\geq1$,
we prove global-in-time existence for sufficiently large $\nu$ when
$0\leq\rho_0\in L^\infty(\mathbb T^2)$ and
$(u_0,w_0)\in H^1(\mathbb T^2)\times H^1(\mathbb T^2)$. Apart from the
well-preparedness condition
\[
\sqrt\nu\,\|\operatorname{div}u_0\|_{L^2}\leq\digamma,
\]
no smallness is imposed on the $H^1$ size of the initial velocity or
microrotation; the required threshold for $\nu$ may depend on the size
of the data. For the linear pressure law, the resulting solution is
unique in the class stated below. Second, for fixed initial data with
$\operatorname{div}u_0=0$, we justify the limit $\nu\to\infty$ toward
\begin{equation}\label{Incom}
\begin{cases}
\eta_t+v\cdot\nabla\eta=0,\\
\eta(v_t+v\cdot\nabla v)-(\mu+\xi)\Delta v+\nabla\Pi
=-2\xi\nabla^\perp w,\\
\eta(w_t+v\cdot\nabla w)+4\xi w
=\mu_2\Delta w+2\xi\nabla^\perp\!\cdot v,\\
\operatorname{div}v=0,
\end{cases}
\qquad (t,x)\in\mathbb R_+\times\mathbb T^2,
\end{equation}
with initial data $(\eta,v,w)|_{t=0}=(\rho_0,u_0,w_0)$. Here $\Pi$
denotes the incompressible pressure. The convergence is subsequential in
the weak and strong topologies specified in Theorem~\ref{result2}. A
quantitative consequence of the uniform estimates is
\[
\|\operatorname{div}u^\nu\|_{L^\infty(\mathbb R_+;L^2)}
+\|\operatorname{div}u^\nu\|_{L^2(\mathbb R_+\times\mathbb T^2)}
\leq C\nu^{-1/2}.
\]

\bigbreak

\noindent\textbf{Notation and
conventions.} 

Before stating the main results, we fix the notation used below. The
two-dimensional torus is normalized so that $|\mathbb T^2|=1$. Unless
another domain is displayed, all spatial integrals and norms are taken
over $\mathbb T^2$. We write $\|\cdot\|_{L^q}$ for the norm of
$L^q(\mathbb T^2)$ and set $H^k=W^{k,2}$. Standard notation is used
for Bochner spaces on time intervals.
For an integrable function $h$, its mean and zero-mean part are denoted by
\[
\bar h:=\int_{\mathbb T^2}h(x)\,dx,
\qquad \widetilde h:=h-\bar h.
\]
The inverse $(-\Delta)^{-1}$ always acts on zero-mean functions. The
Helmholtz projections onto solenoidal and potential vector fields are
denoted by $\mathcal P$ and $\mathcal Q:=I-\mathcal P$, respectively.
For a scalar or vector field $h$, we use
\[
D_t h=\dot h:=\partial_t h+u\cdot\nabla h.
\]
Moreover, $\nabla^\perp h=(-\partial_2h,\partial_1h)$ and
$\nabla^\perp\!\cdot u=\partial_1u_2-\partial_2u_1$. If $\mathcal{A} = (\mathcal{A}_{ij})$ and $\mathcal{B} = (\mathcal{B}_{ij})$
are matrices, then $\mathcal{A} : \mathcal{B} = \sum_{i,j}\mathcal{A}_{ij}\mathcal{B}_{ij}$. The letter $C$
denotes a positive constant which may vary from one occurrence to the
next; any relevant dependence will be indicated.
The pressure potential associated with $P$ is
\begin{equation}\label{potential-energy}
e(\rho):=\rho\int_1^\rho\frac{P(s)-P(1)}{s^2}\,ds.
\end{equation}
Thus $e(1)=e'(1)=0$ and $\rho e''(\rho)=P'(\rho)$. For
$P(\rho)=\rho^\gamma$,
\[
e(\rho)=
\begin{cases}
\rho\log\rho+1-\rho, & \gamma=1,\\[2mm]
\dfrac{\rho^\gamma-\gamma\rho+\gamma-1}{\gamma-1}, & \gamma>1.
\end{cases}
\]
In particular, for every fixed $R>0 $ and $0\leq\rho\leq R$,
$e(\rho)$ is comparable to $|\rho-1|^2$, with constants depending only
on $R$ and $\gamma$.
We shall use the energy
\[
\mathcal E(t):=\int_{\mathbb T^2}
\left(\frac12\rho|u|^2+\frac12\rho|w|^2+e(\rho)\right)dx,
\qquad \mathcal E_0:=\mathcal E(0).
\]
Every sufficiently regular solution satisfies
\begin{align}\label{energy}
\mathcal E(t)&+\int_0^t\Bigl(
\mu\|\nabla^\perp\!\cdot u\|_{L^2}^2
+\nu\|\operatorname{div}u\|_{L^2}^2
+\mu_2\|\nabla w\|_{L^2}^2
+\xi\|\nabla^\perp\!\cdot u-2w\|_{L^2}^2\Bigr)d\tau
=\mathcal E_0.
\end{align}

\bigbreak

We now state our main theorems concerning the global existence, uniqueness for the linear pressure law, and the large-bulk-viscosity limit for the Cauchy problem \eqref{0}--\eqref{eq:IC}.
\begin{thm}\label{result1}
Let $\digamma>0$ be a given  constant  and let the initial data $(\rho_0,u_0,w_0)$ satisfy 
\begin{equation}\label{a1}
\begin{aligned}
0\leq \rho_0(x)\leq \|\rho_0\|_{L^\infty},\quad \|\operatorname{div} u_0\|_{L^2}\leq \digamma\nu^{-\frac{1}{2}}. 
\end{aligned}
\end{equation}
Assume further that $(u_0,w_0)\in H^1(\mathbb{T}^2)\times H^1(\mathbb{T}^2)$. Then there exists a  number $\nu_0>0$, depending only on  $\digamma,\ \gamma,\ \mu,\ \mu_2,\ \xi, \ \mathcal{E}_0,\ \|\rho_0\|_{L^{\infty}},\ \|\nabla u_0\|_{L^2}$  and $\|\nabla w_0\|_{L^2}$, such that for all $\nu \ge \nu_0$\textsuperscript{1},
\footnotetext[1]{Here we find
$\nu_0 = \max\Bigl\{\mu, C\sqrt{\frac{\rho^* \log(e+\rho^*)}{\mu}} P(\rho^*), \frac{P(\rho^*)}{2}, 8\rho^*+4\sqrt{2\rho^*(2\rho^*+\Lambda(\rho^*))}, \frac{16C_0^4}{\log^4 2}\Bigr\}$} 
 the Cauchy problem \eqref{0}--\eqref{eq:IC} admits a global-in-time solution $(\rho,u,w)$ verifying \eqref{con}--\eqref{norm}, \eqref{energy},
 \begin{align*}
 0\leq \rho(t,x) \leq 2e^{\frac{\gamma-1}{\gamma}\mathcal{E}_0 }\|\rho_0\|_{L^\infty},\quad (t,x)\in \mathbb{R}_+\times\mathbb{T}^2,
 \end{align*}
 \begin{equation}\label{1.11}
\begin{cases}
\rho \in C(\mathbb{R}_+; L^p(\mathbb{T}^2)), \ 1 < p < \infty, \quad (\sqrt{\rho}u, \sqrt{\rho}w) \in C(\mathbb{R}_+; L^2(\mathbb{T}^2)); \\[6pt]
 (u, w) \in L^\infty(\mathbb{R}_+; H^1(\mathbb{T}^2)), 
\ (\nabla(\nabla^{\perp}\cdot u), \nabla \mathcal{G}, \Delta w, \sqrt{\rho} \dot{u}, \sqrt{\rho} \dot{w} ) \in L^2(\mathbb{R}_+; L^2(\mathbb{T}^2));\\[6pt]
(\sqrt{\rho t}  \dot{u}, \sqrt{\rho t}  \dot{w}) \in L^\infty_{\operatorname{loc}}(\mathbb{R}_+; L^2(\mathbb{T}^2)), \ \sqrt{t}(\nabla^{\perp}\cdot \dot{u}, \dot{Z},\nabla \dot{w}) \in L^2_{\operatorname{loc}}(\mathbb{R}_+; L^2(\mathbb{T}^2)); \\[6pt]
(\mathrm{div}\, u, \nabla^{\perp}\cdot u) \in L^r_{\operatorname{loc}}(\mathbb{R}_+; L^\infty(\mathbb{T}^2)), \quad r \in (1, 2).
\end{cases}
\end{equation}
If $\gamma=1$, then the solution is unique. Here, $\dot{u} = u_t + u \cdot \nabla u$ is the material derivative, $\mathcal{G} = \nu \mathrm{div}\, u-P(\rho)$ is the effective viscous flux, and $Z=\operatorname{div}u$.
\end{thm}

The uniform a priori estimates derived in Theorem \ref{result1} allow us to rigorously justify the large bulk-viscosity limit.  They yield subsequential convergence to the
inhomogeneous incompressible micropolar system \eqref{Incom}. The precise statement
is as follows. 
\begin{thm}[Large bulk-viscosity limit]\label{result2}
Fix $\gamma\ge1$ and $\mu,\mu_2,\xi>0$. Assume that
$
0\le\rho_0\in L^\infty(\mathbb T^2)$, $(u_0,w_0)\in H^1(\mathbb T^2)\times H^1(\mathbb T^2)$, and
\[
\int\rho_0\,dx=1,\qquad \int\rho_0u_0\,dx=0,
\qquad \operatorname{div}u_0=0.
\]
For each $\nu\ge\nu_0$, let $(\rho^\nu,u^\nu,w^\nu)$ be a global
solution supplied by Theorem~\ref{result1}. Then every sequence
$\nu_n\to\infty$ contains a subsequence, still denoted by $\nu_n$,
and functions $(\eta,v,w,\Pi)$ such that, for every $T>0$,
\begin{equation*}
\left\{
\begin{aligned}
&\rho^{\nu_n}\overset{*}{\rightharpoonup}\eta
&&\text{in }L^\infty((0,T)\times\mathbb T^2),\\
&(u^{\nu_n},w^{\nu_n})\overset{*}{\rightharpoonup}(v,w)
&&\text{in }L^\infty(0,T;H^1(\mathbb T^2)),\\
&(u^{\nu_n},w^{\nu_n})\to(v,w)
&&\text{in }L^2_{\text{loc}}(0,T;L^2(\mathbb T^2)),\\
&\nabla\mathcal G^{\nu_n}\rightharpoonup-\nabla\Pi
&&\text{in }L^2(0,T;L^2(\mathbb T^2)).
\end{aligned}
\right.
\end{equation*}
The limit $(\eta,v,w,\Pi)$ solves \eqref{Incom} with initial
data $(\rho_0,u_0,w_0)$. Moreover, for a constant $C$ independent
of $\nu$,
\begin{align}\label{div2.11}
\|\operatorname{div}u^\nu\|_{L^\infty(\mathbb R_+;L^2)}
+\|\operatorname{div}u^\nu\|_{L^2(\mathbb R_+\times\mathbb T^2)}
\le C\nu^{-1/2}.
\end{align}
\end{thm}

\subsection{Difficulties and ideas}
We now summarize the principal analytical difficulties and the ideas
used to overcome them. The first issue is that the large coefficient
$\nu$ controls only $\operatorname{div}u$, whereas the transverse
velocity remains coupled to the microrotation. Consequently, the
velocity and microrotation equations cannot be estimated independently.
At the basic level, adding their energy identities produces the
coercive dissipation
\[
\mu\|\nabla^\perp\!\cdot u\|_{L^2}^2
+\mu_2\|\nabla w\|_{L^2}^2
+\xi\|\nabla^\perp\!\cdot u-2w\|_{L^2}^2.
\]
After differentiating along the flow, the same cancellation yields the
term
\[
\xi\|\nabla^\perp\!\cdot\dot u-2\dot w\|_{L^2}^2.
\]
These two coupled coercive structures control the modes on which the
bulk viscosity provides no direct gain and form the main modification
of the corresponding Navier--Stokes energy method.

The second issue is the possible presence of vacuum. The energy method
directly controls $\sqrt\rho\,\dot u$ and
$\sqrt\rho\,\dot w$, but division by $\rho$ is unavailable. Weighted
Poincar\'e inequalities and the two-dimensional logarithmic estimate of
Desjardins \cite{Desjardins1997} recover the unweighted integrability
needed in the nonlinear estimates. To propagate a uniform upper bound
for the density, we introduce the effective viscous flux
\begin{equation}\label{Gdingyi}
\mathcal G:=\nu\operatorname{div}u-P(\rho)
\end{equation}
and, first for strictly positive approximate solutions, the quantity
\[
\Upsilon:=\log\rho-\nu^{-1}(-\Delta)^{-1}
\operatorname{div}(\rho u).
\]
Its transport equation contains a commutator between the velocity and
a double Riesz transform. Combining the effective-flux bounds with a
Hardy-space commutator estimate \cite{Coifman1993} and the logarithmic
interpolation inequality gives an upper bound for $\rho$ that is
independent of the positive lower bound of the approximate density; see
also \cite{R.Danchin2}. Vacuum can therefore be recovered through the
approximation limit, without applying $\log\rho$ to the limiting density.

The third issue is uniformity as $\nu\to\infty$. Let
$Z=\operatorname{div}u$ and $\dot Z=D_tZ$. The identity
\[
\operatorname{div}\dot u
=\dot Z+\partial_i u_j\,\partial_j u_i
\]
shows that estimating $\operatorname{div}\dot u$ as the primary
acoustic quantity would attach the large factor $\nu$ to a quadratic
term for which no such bound is available. We instead estimate
$\dot Z$ and adapt the time-weighted scheme of Liu--Zhang
\cite{liu2026}. The microrotation equation generates additional
commutator terms of the form $[D_t,\Delta]w$. They are controlled by the
uniform $L_t^4L_x^4$ bound for $\nabla u$, elliptic estimates for $w$,
and the coupled material-derivative dissipation displayed above. This
is why the micropolar problem is not obtained by merely appending a
separate estimate for $w$ to the Navier--Stokes argument.

The fourth issue concerns compactness. The estimates are first proved
for smooth solutions with strictly positive density, with constants
independent of its lower bound, and then passed to initial data allowing
vacuum. For the singular limit, $u^\nu$ and $w^\nu$ must be compactified
simultaneously. The time weights degenerate at $t=0$; hence the
Aubin--Lions lemma is applied on $(\delta,T)$, and the interval
$(0,\delta)$ is handled by the uniform $L_t^\infty L_x^2$ bounds. This
yields strong convergence of both fields in $L^2(0,T;L^2)$, which is
sufficient, together with weak-* compactness of the density, to pass to
the nonlinear terms. The weak limit of $\nabla\mathcal G^\nu$ supplies
the incompressible pressure. No strong convergence of the density and
no convergence rate for the full triple are asserted.

Finally, uniqueness for $P(\rho)=\rho$ requires a coupled argument. A
negative Sobolev estimate for the density difference is combined with
a backward vector--scalar parabolic system for the velocity and
microrotation differences. Hardy--BMO duality, logarithmic
interpolation, and Osgood's lemma then close the estimate
\cite{Bahouri2011,Mucha2010}. The coupling in the backward system is
essential: using only the dual problem from the Navier--Stokes case
would leave the microrotation difference uncontrolled.

\bigbreak

The remainder of the paper is organized as follows. Section~2 collects
the interpolation, weighted Poincar\'e, Hardy--BMO, and elliptic
estimates used below. Section~3 establishes the basic Sobolev estimate,
the uniform upper bound for the density, and the time-weighted estimates
for smooth strictly positive solutions. Section~4 passes to initial
densities with vacuum, proves uniqueness for the linear pressure law,
and justifies the large-bulk-viscosity limit.

\section{Preliminaries and key lemmas}

In this section, we collect several fundamental analytical tools that will be used repeatedly throughout the paper, including standard interpolation inequalities, weighted Poincar\'{e}-type estimates, and properties of Hardy and $\mathrm{BMO}$ spaces on the 2D torus.

We first recall the  Gagliardo-Nirenberg interpolation inequalities on $\mathbb{T}^2$ (see \cite{Nirenberg1959}).
\begin{lem}\label{GNS}
Let $2\leq p<\infty$, $1<q<\infty$ and $2<r<\infty$. For any
$g \in H^1(\mathbb{T}^2)$ and $h \in L^q (\mathbb{T}^2)\cap W^{1,r}(\mathbb{T}^2)$, there exists a positive constant $C$  depending on $p,q,r$ such that
\begin{align}\label{ding2.1}
\|g\|_{L^p} \leq C\|g\|_{L^2}^{\frac{2}{p}}\|\nabla g\|_{L^2}^{\frac{p-2}{p}} + C\|g\|_{L^2},
\end{align}
and
\begin{align}\label{ding2.2}
\|h\|_{L^\infty} \leq C\|h\|_{L^q}^{\frac{q(r-2)}{2r+q(r-2)}}\|\nabla h\|_{L^r}^{\frac{2r}{2r+q(r-2)}} + C\|h\|_{L^2}.
\end{align}
In addition, if $\bar{g} = 0$ or  $\bar{h} = 0$, the last terms on the right-hand side of \eqref{ding2.1}-\eqref{ding2.2} may be omitted.
\end{lem}

To bound $L^p$-norms of the velocity field, we employ a Poincar\'{e}-type inequality from \cite{liu2026}.
\begin{lem}\label{lemma2.2}
Assume that $\rho\geq 0$ satisfies, for some constants $M_1$ and $M_2$ and $\theta>1$,
\begin{align}\label{ding2.3}
0 < M_1 \leq \int \rho(x) dx, \quad \int \rho^\theta(x) dx \leq M_2.
\end{align}
Then for any $p \geq 2$, there exists a  constant $C(M_1,M_2, \theta,p)>0$  such that for all $u \in H^1(\mathbb{T}^2)$,
\begin{align}\label{ding2.4}
\|u\|_{L^p}^2 \leq C \left( \|\sqrt{\rho}u\|_{L^2}^2 + \|\nabla u\|_{L^2}^2 \right).
\end{align}
We note that condition \eqref{ding2.3} is automatically satisfied for barotropic pressure laws $P(\rho) = \rho^\gamma$ with $\gamma \geq 1$, provided the density is uniformly bounded on $\mathbb{T}^2$.
\end{lem}

We now recall the classical \emph{div-curl} estimate (see \cite{Bourguignon1974}), which allows one to bound the full gradient of a vector field using its divergence and scalar vorticity.
\begin{lem}\label{lemma2.3}
Let $k \in \mathbb{N}$ and $1 < p < \infty$. For any $u \in W^{k+1,p}(\mathbb{T}^2)$, there exists a constant $C = C(k, p)>0$ such that
\begin{align}\label{ding2.5}
    \|\nabla u\|_{W^{k,p}} \leq C\left( \|\operatorname{div} u\|_{W^{k,p}} + \|\nabla^{\perp}\cdot u\|_{W^{k,p}} \right).
\end{align}
 \end{lem}

The following Beale-Kato-Majda type inequality (see \cite{Huang2011,Kato1986}) will be used to obtain uniform $W^{1,q}$ estimates for the density and the time-integrability
of $\|\nabla u\|_{L^\infty}$.
\begin{lem}
    Let $2< q < \infty$, and suppose $\nabla u \in W^{1,q}$. Then there exists a constant $C=C(q)>0$ such that
    \begin{equation} \label{BKMineq}
        \begin{aligned}
            \lnorm{\nabla u}{\infty} \le C(\lnorm{\otv u}{\infty}+\lnorm{\crl \cdot u}{\infty})\log \left(e+ \lnorm{\nabla^2 u}{q}\right)+C\lnorm{\nabla u}{2}+C.
        \end{aligned}
    \end{equation}
\end{lem}

Finally, we recall the definitions of the $\mathcal{H}^1(\mathbb{T}^2)$  and  $\mathrm{BMO}$ spaces on the torus. Let $(-\Delta)^{-1}$ denote the inverse Laplacian  with zero spatial mean, and let
$\mathcal{R}_i := (-\Delta)^{-1/2}\partial_i$ be the standard Riesz transforms. Following \cite{Coifman1975}, we define
$$
\mathcal{H}^1(\mathbb{T}^2) := \left\{ h \in L^1(\mathbb{T}^2) : \|{h}\|_{\mathcal{H}^1} := \|{h}\|_{L^1} + \|\mathcal{R}_{1}h\|_{L^1}+\|\mathcal{R}_2h\|_{L^1}<\infty,\, \bar{h}=0\right\},
$$
and
\[
\mathrm{BMO}(\mathbb{T}^2) := \left\{ h \in L^1_{\text{loc}}(\mathbb{T}^2) : \|{h}\|_{\mathrm{BMO}} < \infty \right\},
\]
where
\[
\|h\|_{\mathrm{BMO}} := \sup_{\substack{x\in\mathbb{T}^2, r\in(0,1)}} \frac{1}{|\Omega_r(x)|} \int_{\Omega_r(x)} \left| h(y) - \frac{1}{|\Omega_r(x)|} \int_{\Omega_r(x)} h(z)\mathrm{d}z \right| \mathrm{d}y,
\]
with $\Omega_r(x) = \mathbb{T}^2 \cap B_r(x)$ and $B_r(x)$ the ball of radius $r$ centered at $x$.
It is well known that $\mathcal{H}^1(\mathbb{T}^2)$ and $\mathrm{BMO}(\mathbb{T}^2)$ are dual to each other
\cite{Coifman1975}. We also have the following result from \cite{Coifman1993,Mucha2010}.
\begin{lem}
    Let $u, v \in L^2(\mathbb{T}^2)$ satisfy $\nabla^\perp\cdot u = 0$ and $\mathrm{div}\,v = 0$ in the sense of distributions. Then  $u \cdot v \in \mathcal{H}^1(\mathbb{T}^2)$ and
\begin{align}\label{BMO}
\|u \cdot v\|_{\mathcal{H}^1} \leq \|u\|_{L^2}\|v\|_{L^2}.
\end{align}
\end{lem}

In what follows, we also recall the following lemmas from  \cite{R.Danchin2,Desjardins1997}.
\begin{lem}\label{lem:desjardins}
Let $\rho\in L^\infty(\mathbb T^2)$ satisfy $\rho\geq0$ and
 $M:=\int_{\mathbb T^2}\rho\,dx>0$.
Let $u\in H^1(\mathbb T^2)$ satisfy 
\[
\int_{\mathbb T^2}\rho u\,dx=0,
\]
which can be obtained from \eqref{con} and \eqref{norm}. Then there exists a universal constant $C>0$ such that, for every
$c\in\mathbb R$,
\begin{equation}\label{A.9}
\left(\int_{\mathbb T^2}\rho|u|^4\,dx\right)^{\frac{1}{2}}
\leq C\|\sqrt\rho u\|_{L^2}\|\nabla u\|_{L^2}
\log^{\frac{1}{2}}\left(
e+\frac{\|\rho-c\|_{L^2}}{M}
+\frac{\|\rho\|_{L^2}\|\nabla u\|_{L^2}^2}
{\|\sqrt\rho u\|_{L^2}^2}
\right).
\end{equation}
In particular, if $M=1$, then, with $c=1$ and
$\widetilde\rho:=\rho-1$,
\begin{align}\label{Desjardins1997}
\left(\int_{\mathbb T^2}\rho|u|^4\,dx\right)^{\frac{1}{2}}
\leq C\|\sqrt\rho u\|_{L^2}\|\nabla u\|_{L^2}
\log^{\frac{1}{2}}\left(
e+\|\widetilde\rho\|_{L^2}
+\frac{\|\rho\|_{L^2}\|\nabla u\|_{L^2}^2}
{\|\sqrt\rho u\|_{L^2}^2}
\right).
\end{align}
If $\|\sqrt\rho u\|_{L^2}=0$, the estimate is understood by
continuity and is immediate.
\end{lem}

\begin{lem}
Let $\rho \in L^{p'}(\mathbb{T}^2)$ with
$\frac{1}{p} + \frac{1}{p'} = 1$
and $2 \leq p < +\infty$.
Suppose  $b\in H^1(\mathbb{T}^2)$ satisfies
\begin{equation}
\int \rho b \,dx = 0 \quad \text{and} \quad M := \int \rho \, dx > 0.
\end{equation}
Then there exists  $C_p>0$ (with $C_2 = 1$) such that for any $c\in \mathbb{R}$,
\begin{equation}\label{poincare}
\|b\|_{L^2} \leq \left(1 + \frac{C_p}{M} \|\rho - c\|_{L^{p'}}\right) \|\nabla b\|_{L^2}.
\end{equation}
Moreover,
\begin{equation}\label{A.3}
\|b\|_{L^2} \leq C \log^{\frac{1}{2}} \left( e + \frac{\|\rho - c\|_{L^2}}{M} \right) \|\nabla b\|_{L^2}.
\end{equation}
\end{lem}

\bigbreak

\section{Regularity estimates}
\subsection{Sobolev estimates for the velocity and  microrotation}

We establish uniform Sobolev estimates for
 the velocity field $u$ and the micro-rotation  velocity $w$ of 
solutions $(\rho, u, w)$ on $\mathbb{R}_+\times \mathbb{T}^2$. A key feature of our analysis is that all such estimates depend exclusively on a uniform upper bound of the density. Henceforth, we denote $\tilde{P} \coloneqq P - \bar{P}$ and $\tilde{\mathcal{G}} \coloneqq \mathcal{G}- \bar{\mathcal{G}}$, where $\bar{P}$ and $\bar{\mathcal{G}}$ are the spatial averages over  the unit torus  $\mathbb{T}^2$, respectively. We further let $\mathcal{P}$ denote  the Helmholtz projection onto the divergence-free vector fields, $\rho^*$ denote  the upper bound of the density, and define  $P^* := {\max_{0\le r\le\rho^*}P(r)}$. Recalling the definition of \emph{effective viscous flux},  we have
\begin{equation}\label{G}
\tilde{\mathcal{G}} = \nu \operatorname{div} u - \tilde{P}.
\end{equation}
 
Throughout this section, we assume that the pressure $P(\rho)$ satisfies  the $\gamma$-law   and that the density obeys the uniform bound 
\begin{equation}\label{shangquejie}
    \begin{aligned}
        0\leq\rho(t,x)\leq\rho^* \quad \text{for all}\, (t,x)\in[0,T]\times\mathbb{T}^2.
    \end{aligned}
\end{equation}
We further denote the material derivative of $u$ by $\dot{u}$, given by $\dot{u}:= u_t+u\cdot\nabla u$.

\begin{prop}\label{prop1}
Let $(\rho,\,u,\,w)$ be a smooth solution of the Cauchy problem \eqref{0}--\eqref{eq:IC} on $[0,T]\times\mathbb{T}^2$ satisfying  \eqref{con}--\eqref{norm}. Assume the pressure satisfies   $P(\rho) = \rho^{\gamma}$ and   set 
$$\Lambda:= \rho P'(\rho)-P(\rho)=(\gamma-1)P(\rho).$$
Then there exists: \\
-- a functional  $\Phi$ satisfying
\begin{align}\label{3.4}
    \Phi \geq& \frac{1}{2}\int \left(\rho |u|^2+\rho|w|^2+\mu|\nabla^{\perp}\cdot u|^2+\frac{1}{\nu}(\tilde{\mathcal{G}}^2+\tilde{P}^2)\right)dx \nonumber \\
    &+\frac{1}{2}\int\left(\mu_2|\nabla w|^2+\xi|\nabla^{\perp}\cdot u-2w|^2+2e\right)d x,
\end{align}
-- an absolute constant $C>0$,\\
-- a positive constant $\nu_0$ depending only on  $P$,  $\mu$,  $\mu_2$, $\xi$, and $\rho^*$, such that for all $\nu\geq\nu_0$ \textsuperscript{2},
\footnotetext[2]{Here we find
$\nu_0=\max \{\mu,C\sqrt{\frac{\rho^*\log (e+\rho^*)}{\mu}}P^*,8\rho^*+4\sqrt{2\rho^*(2\rho^*+\Lambda(\rho^*))},\frac{P^*}{2}\}$.}
  and all $t\in[0,T]$,
\begin{align}\label{3.5}
  1+\frac{1}{\mu\mathcal E_0}
\left(
\Phi(t)+\int_0^t\mathcal D(\tau)\,d\tau
\right)
\leq{}&
\exp\Bigg\{
\Bigg[
\log\left(1+\frac{\Phi_0}{\mu\mathcal E_0}\right)
+
C\left(
1+\frac{(\rho^*)^2}{\mu_3^3}\mathcal E_0
\log(e+\rho^*)
\right)\frac{\mathcal E_0}{\mu}
\nonumber\\
&+
\frac{2\xi^2}{\rho^*\mu}
\left(
\frac1{\mu_2}+\frac3\mu+\frac2\xi
\right)
\Bigg]
\exp\left(
C\frac{(\rho^*)^2\mathcal E_0^2}{\mu^4}
\right)
\Bigg\},
\end{align}
where \(\mu_3:=\min\{\mu,\mu_2\}\), $ \mathcal{E}_0$  is the initial energy given by \eqref{energy}, and the dissipation functional $\mathcal{D}$ is given by
\begin{align}
    \mathcal{D} := &
\frac{1}{4}\|\sqrt{\rho} \dot{u}\|_{L^2}^2 +\frac{1}{4}\|\sqrt{\rho} \dot{w}\|_{L^2}^2+ \frac{1}{12\rho^*}\mu_2^2\|\Delta w\|_{L^2}^2+\frac{(\mu+\xi)^2}{8\rho^*} \|\nabla(\nabla^{\perp}\cdot u)\|_{L^2}^2 + \frac{1}{16\rho^*} \|\nabla \mathcal{G}\|_{L^2}^2
 \nonumber \\
&+ \frac{1}{4\nu} \|\tilde{P}\|_{L^2}^2+\int \left( \frac{\nu + \Lambda}{2} \right) (\operatorname{div } u)^2dx + \frac{\mu+2\xi}{2} \|\nabla^{\perp}\cdot u\|_{L^2}^2+\frac{\mu}{4}\|\nabla u\|_{L^2}^2+\mu_2\|\nabla w\|_{L^2}^2.
\end{align}

\end{prop}

\begin{proof}
The overall strategy follows the approach developed by Danchin--Mucha \cite{R.Danchin2}, and relies on constructing a modified energy functional that captures the $H^1$ regularity of both the velocity and the micro-rotational field, combined with the logarithmic interpolation inequality \eqref{Desjardins1997}.

\bigbreak

We begin by deriving bounds for the quantity
\begin{align}\label{Phi}
    \tilde{\Phi}=\frac{1}{2}\int\left(\mu|\nabla^{\perp}\cdot u|^2+\mu_2|\nabla w|^2+\xi|\nabla^{\perp}\cdot u-2w|^2+\frac{1}{\nu}\left(\tilde{P}^2+\tilde{\mathcal{G}}^2+\rho\int_{\rho}^1\frac{P^2(\tau)}{\tau^2}d\tau\right)\right)dx.
\end{align}
Taking the $ L^2 $ inner product of the momentum equation \eqref{0}$_2$ with the material derivative $\dot{u}$ gives
\begin{align}\label{u_1}
    \int \rho |\dot{u}|^2dx&+\frac{1}{2}\frac{d}{dt}\int \left((\mu+\xi)|\nabla u|^2+(\mu+\lambda-\xi)(\operatorname{div}u)^2\right)\,dx +\int \nabla P\cdot u_t\, dx\nonumber\\
    =&\int (\rho \dot{u})\cdot (u\cdot \nabla u)dx-2\xi\int \nabla^{\perp}w\cdot u_tdx\nonumber\\
     =&\int (\rho \dot{u})\cdot (u\cdot \nabla u)dx+2\xi\int  w(\nabla^{\perp}\cdot u_t)dx.
\end{align}
Similarly, testing the angular momentum equation $\eqref{0}_3$ with $\dot{w}$ yields
\begin{align}\label{w_1}
    \int\rho |\dot{w}|^2dx+\frac{d}{dt}\int \left(\frac{1}{2}\mu_2|\nabla w|^2 + 2\xi|w|^2\right)dx
    =\int(\rho \dot{w})\cdot (u\cdot \nabla w)dx+2\xi\int(\nabla^{\perp}\cdot u)  w_t \, dx.
\end{align}
Then, combining \eqref{u_1} and \eqref{w_1}, we obtain
\begin{align}\label{u_1+w_1}
    \int(\rho |\dot{u}|^2&+\rho |\dot{w}|^2)dx+\frac{1}{2}\frac{d}{dt}\int\left((\mu+\xi)|\nabla u|^2+\mu_2|\nabla w|^2+(\mu+\lambda-\xi)(\operatorname{div}u)^2+4\xi|w|^2\right)dx\nonumber\\
   &+\int\nabla P\cdot u_t dx 
   =\int(\rho \dot{u})\cdot (u\cdot \nabla u)dx+\int(\rho \dot{w})\cdot (u\cdot \nabla w)dx+2\xi\frac{d}{dt}\int(\nabla^{\perp}\cdot u)  w dx.
\end{align}
Using the algebraic identity
$$\int\left(|\nabla^{\perp}\cdot u|^2+4|w|^2-4w(\nabla^{\perp}\cdot{u})\right)dx=\int|\nabla^{\perp}\cdot u-2w|^2dx, $$
we may rewrite \eqref{u_1+w_1}  as
 \begin{align}\label{3.7}
      \int(\rho |\dot{u}|^2&+\rho |\dot{w}|^2)dx+\frac{1}{2}\frac{d}{dt}\int\left(\mu|\nabla^{\perp}\cdot u|^2+\mu_2|\nabla w|^2+\nu(\operatorname{div}u)^2+\xi|\nabla^{\perp}\cdot u-2w|^2\right)dx \nonumber\\
      &+\int\nabla P\cdot u_t dx
    =\int(\rho \dot{u})\cdot (u\cdot \nabla u)dx+\int(\rho \dot{w})\cdot (u\cdot \nabla w)dx,
 \end{align}
where   $\nu=2\mu+\lambda$. For the pressure term on the left-hand side of \eqref{3.7}, we have
\begin{align}\label{wanneng}
   \partial_t P + \operatorname{div}(Pu) + \Lambda \operatorname{div}u = 0,
\end{align}
with $\Lambda:=\rho P'-P$. As in \cite{R.Danchin2}, one can derive
\begin{align*}
\int \nabla P \cdot u_t \,dx
 = &\frac{d}{dt}\int \left( \frac{k -(\bar{P})^2}{\nu} - P\operatorname{div}u \right) dx  - \frac{2}{\nu}\int \tilde{\mathcal{G}} \operatorname{div}u\, \Lambda \,dx\\
  &+ \int (\operatorname{div}u)^2 \Lambda \,dx + \frac{1}{\nu}\int Pu \cdot \nabla \mathcal{G} \,dx,
\end{align*}
where the auxiliary function  $k=k(\rho) = P^2(\rho) - \frac{\rho}{2} \int_{1}^{\rho} \frac{P^2(\varrho)}{\varrho^2} \mathrm{d}\varrho$ solves the ODE
    \[
    k - \rho k' = -\frac{P^2}{2} - 2P \Lambda(\rho) \quad \text{with} \quad k(1) = P^2(1).
    \]
We now introduce the following energy functional
\begin{align}\label{3.12}
    \check{E} \coloneqq \int  \left(
\frac{\mu}{2} |\nabla^{\perp}\cdot u|^2 + \frac{\nu}{2} (\operatorname{div}u)^2 +\frac{\mu_2}{2} |\nabla w|^2+ \frac{1}{2}\xi|\nabla^{\perp}\cdot u-2w|^2+\frac{1}{\nu} \left( k - (\bar{P})^2 \right) - \tilde{P}\operatorname{div}u
\right) dx.
\end{align}
Then combining \eqref{u_1} and \eqref{3.12}  yields
\begin{align}\label{3.13}
\frac{d}{dt}\check{E}+ \int (\rho|\dot u|^2 +\rho|\dot w|^2)d x
+ \int (\operatorname{div}u)^2 \Lambda dx= & \int {(\rho\dot{u})} \cdot (u \cdot \nabla u) dx+ \int {(\rho\dot{w})} \cdot (u \cdot \nabla w) dx
\nonumber \\
&+ \frac{2}{\nu}\int \tilde{\mathcal{G}} \operatorname{div}u \Lambda dx
- \frac{1}{\nu}\int P u \cdot \nabla \mathcal{G} dx.
\end{align}
The equivalence $\tilde{\Phi} \simeq \check{E}$ follows from direct algebraic manipulation. Indeed,
\begin{align*}
\check{E} = &\int \left(
\frac{\mu}{2} |\nabla^{\perp}\cdot u|^2
+\frac{\mu_2}{2} |\nabla w|^2+ \frac{1}{2}\xi|\nabla^{\perp}\cdot u-2w|^2\right)d x\\
&+\int  \frac{1}{2\nu} \left( (\nu \operatorname{div}u)^2 - 2\nu \operatorname{div}u \tilde{P} + 2\left(k - (\bar{P})^2\right) \right)
  dx.
\end{align*}
To control the right-hand side of \eqref{3.13}, we rewrite the momentum equation using  $\mathcal{G} = \nu \operatorname{div} u - P$:
\begin{equation}\label{3.15}
(\mu+\xi) (\Delta u - \nabla \operatorname{div} u) + \nabla \mathcal{G} = \rho \dot{u}+2\xi\nabla^{\perp}w.
\end{equation}
Standard elliptic estimates from the Hodge decomposition give
\begin{align}\label{3.16}
    (\mu+\xi)^2 \|\Delta \mathcal{P} u\|^2_{L^2} + \|\nabla \mathcal{G}\|^2_{L^2} = \left\| \rho \dot{u}+2\xi \nabla^{\perp}w \right\|^2_{L^2}\leq 2\rho^* \left\| \sqrt{\rho} \dot{u} \right\|^2_{L^2}+8\xi^2\|\nabla^{\perp}w\|_{L^2}^2.
\end{align}
Similarly, the angular momentum equation becomes
$$\mu_2\Delta w=\rho\dot{w}+4\xi w-2\xi\nabla^{\perp}\cdot u,$$
so that
\begin{align*}
\mu_2^2\|\Delta w\|_{L^2}^2=\|\rho \dot{w}+4\xi w-2\xi\nabla^{\perp}\cdot u\|_{L^2}^2
\leq 3\left(\rho^*\|\sqrt{\rho} \dot{w}\|_{L^2}^2+16\xi^2\|w\|_{L^2}^2+4\xi^2\|\nabla^{\perp}\cdot u\|_{L^2}^2\right).
\end{align*}
By Young's inequality,
$$
\frac{2}{\nu} \int \tilde{\mathcal{G}} \operatorname{div} u \Lambda  dx
\leq \frac{1}{2} \int (\operatorname{div} u)^2 \Lambda   dx
+ \frac{2}{\nu^2} \int \tilde{\mathcal{G}}^2\Lambda  dx.$$
Combining the above with $\tilde{\Phi}=\check{E}$ and the important identity
\begin{align}\label{tihuan}
    \|\nabla^m(\nabla^{\perp}\cdot u)\|^2_{L^2}=\|\nabla^m(\nabla ^{\perp}\cdot\mathcal{P}u)\|^2_{L^2}, \quad  \quad m=0,1,
\end{align}
 we  conclude
\begin{align*}
&\frac{d}{dt} \tilde{\Phi}
+ \frac{1}{4} \left\| \sqrt{\rho} \dot{u} \right\|^2_{L^2}
+ \frac{(\mu+\xi)^2}{4\rho^*} \left\| \nabla(\nabla^{\perp}\cdot u)\right\|^2_{L^2}
+ \frac{1}{4\rho^*} \left\| \nabla \mathcal{G} \right\|^2_{L^2}
  \\
&+ \frac{1}{4} \left\| \sqrt{\rho} \dot{w} \right\|^2_{L^2}+\frac{\mu_2^2}{6\rho^*} \|\Delta w\|^2_{L^2}+ \frac{1}{2} \int  (\operatorname{div} u)^2 \Lambda \, dx \\
\leq&
\frac{2}{\nu^2} \int \tilde{\mathcal{G}}^2\Lambda \, dx
- \frac{1}{\nu} \int  P\,u \cdot \nabla \mathcal{G} \, dx
+ \left\| \sqrt{\rho} u \cdot \nabla w \right\|^2_{L^2}\\
&+\frac{2\xi^2}{\rho^*}\|\nabla^{\perp} w\|^2_{L^2}+\frac{8\xi^2}{\rho^*}\|w\|^2_{L^2}+\frac{2\xi^2}{\rho^*}\|\nabla^{\perp}\cdot u\|^2_{L^2}+ \left\| \sqrt{\rho} u \cdot \nabla u \right\|^2_{L^2}.
\end{align*}
To estimate the convective terms, we use the decomposition
\begin{equation}\label{3.17}
\nabla u = \nabla \mathcal{P} u - \frac{1}{\nu} \nabla^2 (-\Delta)^{-1} \tilde{\mathcal{G}} - \frac{1}{\nu} \nabla^2 (-\Delta)^{-1} \tilde{P},
\end{equation}
which leads to
\begin{align}\label{3.18}
&\frac{d}{dt} \tilde{\Phi}
+ \frac{1}{4} \left\| \sqrt{\rho} \dot{u} \right\|^2_{L^2}
+ \frac{(\mu+\xi)^2}{4\rho^*} \left\| \nabla(\nabla^{\perp}\cdot u)\right\|^2_{L^2}
+ \frac{1}{4\rho^*} \left\| \nabla \mathcal{G} \right\|^2_{L^2}
 \nonumber \\
&\quad+ \frac{1}{4} \left\| \sqrt{\rho} \dot{w} \right\|^2_{L^2}+\frac{\mu_2^2}{6\rho^*} \|\Delta w\|^2_{L^2}+ \frac{1}{2} \int  (\operatorname{div} u)^2\Lambda \, dx \nonumber \\
\leq&
\frac{2}{\nu^2} \int_{\mathbb{T}^2} \tilde{\mathcal{G}}^2\Lambda \, dx
- \frac{1}{\nu} \int_{\mathbb{T}^2} P\,u \cdot \nabla \mathcal{G} \, dx
+ \left\| \sqrt{\rho} u \cdot \nabla w \right\|^2_{L^2}\nonumber  \\
&+\frac{2\xi^2}{\rho^*}\|\nabla^{\perp} w\|^2_{L^2}+\frac{8\xi^2}{\rho^*}\|w\|^2_{L^2}+\frac{2\xi^2}{\rho^*}\|\nabla^{\perp}\cdot u\|^2_{L^2}\nonumber\\
&+ 3\left(\|\sqrt{\rho}u\cdot\nabla \mathcal{P}u\|^2_{L^2}+\frac{1}{\nu^2}\|\sqrt{\rho}u\cdot\nabla^2(-\Delta)^{-1}\tilde{\mathcal{G}}\|^2_{L^2}+\frac{1}{\nu^2}\|\sqrt{\rho}u\cdot\nabla^2(-\Delta)^{-1}\tilde{P}\|^2_{L^2}\right).
\end{align}

\bigbreak

We proceed to bound the right-hand side of  \eqref{3.18}. Following an argument similar to that in \cite{R.Danchin2} and combining  Lemma \ref{GNS},  \eqref{Desjardins1997}, \eqref{A.3} and \eqref{tihuan}, we obtain the following estimates:
\begin{align*} 
    3 \int \rho \bigl| u \cdot \nabla \mathcal{P}u \bigr|^2 dx
\leq &\frac{(\mu+\xi)^2}{8 \rho^*} \|\nabla(\nabla^{\perp}\cdot u)\|_{L^2}^2 \nonumber\\
&+ \frac{C (\rho^*)^2}{(\mu+\xi)^2} \, \bigl\| \sqrt{\rho} u \bigr\|_{L^2}^2 \|\nabla u\|_{L^2}^2 \|\nabla^{\perp}\cdot u\|_{L^2}^2   \log \left( e + \|\tilde{\rho}\|_{L^2}+ \frac{\|\rho\|_{L^2} \|\nabla u\|_{L^2}^2}{\bigl\| \sqrt{\rho}u \bigr\|_{L^2}^2} \right),
\end{align*}
\begin{align*} 
     \int &\rho \bigl| u \cdot \nabla w \bigr|^2 dx\leq \frac{\mu_2^2}{12\rho^*} \|\Delta w\|_{2}^2 + \frac{C (\rho^*)^2}{\mu_2^2} \, \bigl\| \sqrt{\rho} u \bigr\|_{2}^2 \|\nabla u\|_{2}^2 \|\nabla w\|_{2}^2
\log \left( e + \|\tilde{\rho}\|_{2} + \frac{\|\rho\|_{2} \|\nabla u\|_{2}^2}{\bigl\| \sqrt{\rho}u \bigr\|_{2}^2} \right),
\end{align*}
\begin{align*} 
 \frac{3}{\nu^2} \int \rho \left| u \cdot \left[ \nabla^2(-\Delta)^{-1} \tilde{\mathcal{G}} \right] \right|^2 dx
\leq& \frac{C (\rho^*)^2}{\nu^4} \, \bigl\| \sqrt{\rho} u \bigr\|_{L^2}^2 \|\nabla u\|_{L^2}^2 \|\tilde{\mathcal{G}}\|_{L^2}^2
\log \left( e + \|\tilde{\rho}\|_{L^2} + \frac{\|\rho\|_{L^2} \|\nabla u\|_{L^2}^2}{\bigl\| \sqrt{\rho} u \bigr\|_{L^2}^2} \right)\nonumber\\
&+ \frac{1}{16 \rho^*} \|\nabla \mathcal{G}\|_{L^2}^2,
\end{align*}
\begin{align*} 
 \frac{3}{\nu^2} \int \rho \bigl| u \cdot \nabla^2 (-\Delta)^{-1} \tilde{P} \bigr|^2 dx
\leq &  \frac{C (\rho^*)^2}{\nu^3} \, \bigl\| \sqrt{\rho} u \bigr\|_{L^2}^2 \|\nabla u\|_{L^2}^2
\log \left( e + \|\tilde{\rho}\|_{L^2} + \frac{\|\rho\|_{L^2} \|\nabla u\|_{L^2}^2}{\bigl\| \sqrt{\rho} u \bigr\|_{L^2}^2} \right) \|\tilde{P}\|_{L^{\infty}}^2\nonumber\\
&+\frac{1}{4 \nu} \|\tilde{P}\|_{L^2}^2,
\end{align*}
and
\begin{equation*} 
- \frac{1}{\nu} \int  P u\cdot \nabla \mathcal{G} \, dx
\leq  C \frac{\rho^*}{\nu^2} (P^*)^2 \|\nabla u\|_{L^2}^2 \log \bigl( e + \|\tilde{\rho}\|_{L^2} \bigr)+\frac{1}{16 \rho^*} \|\nabla \mathcal{G}\|_{L^2}^2.
\end{equation*}
Therefore,    substituting the above estimates into  \eqref{3.18}, we arrive at
\begin{align}\label{3.23}
    &\frac{d}{dt} \tilde{\Phi} + \frac{1}{4} \left\| \sqrt{\rho} \dot{u} \right\|_{L^2}^2
  + \frac{1}{4} \left\| \sqrt{\rho} \dot{w} \right\|_{L^2}^2+\frac{\mu_2^2}{12\rho^*} \|\Delta w\|_{L^2}^2+ \frac{1}{2} \int  (\operatorname{div} u)^2\Lambda \, dx \nonumber \\
  &\quad+ \frac{(\mu+\xi)^2}{8\rho^*} \left\| \nabla(\nabla^{\perp}\cdot u)\right\|_{L^2}^2
+ \frac{1}{8\rho^*} \left\| \nabla \mathcal{G} \right\|_{L^2}^2 \nonumber \\
\leq&\frac{1}{4\nu}\|\tilde{P}\|_{L^2}^2
+\frac{2\xi^2}{\rho^*}\|\nabla^{\perp} w\|_{L^2}^2+\frac{8\xi^2}{\rho^*}\|w\|_{L^2}^2+\frac{2\xi^2}{\rho^*}\|\nabla^{\perp}\cdot u\|_{L^2}^2+ \frac{C\rho^*}{\nu^2}  \log \bigl( e + \|\tilde{\rho}\|_{L^2}\bigr)(P^*)^2 \|\nabla u\|_{L^2}^2\nonumber  \\
&+\frac{2}{\nu^2} \int \tilde{\mathcal{G}}^2\Lambda \, dx+C\|\sqrt{\rho}u\|_{L^2}^2\|\nabla u\|_{L^2}^2\log\left(e+\|\tilde{\rho}\|_{L^2}+\frac{\|\rho\|_{L^2}\|\nabla u\|_{L^2}^2}{\|\sqrt{\rho}u\|_{L^2}^2}\right)\nonumber\\
&\quad\times \left(\frac{\rho^*\|\tilde{P}\|_{L^{\infty}}^2}{\nu^3}+\frac{(\rho^*)^2}{\nu^4}\|\tilde{\mathcal{G}}\|_{L^2}^2
+\frac{(\rho^*)^2}{(\mu+\xi)^2}\|\nabla^{\perp}\cdot u\|_{L^2}^2+\frac{(\rho^*)^2\|\nabla w\|_{L^2}^2}{\mu_2^2}\right).
\end{align}

\bigbreak

Next, we improve the energy functional $\tilde{\Phi}$. To close the bootstrap argument, we augment $\tilde{\Phi}$ with  the basic energy $\mathcal{E}$ and the internal potential energy $e$. This extension is needed to control the pressure fluctuation $\|\tilde{P}\|_{L^2}$. From the continuity equation, the internal energy $e$ satisfies
\[
\partial_t e + \operatorname{div} (e u) + P \operatorname{div} u = 0.
\]
Integrating over   $\mathbb{T}^2$ and using  $\nu \, \text{div}\, u = \tilde{P} + \tilde{\mathcal{G}}$, we obtain
\begin{align}\label{3.24}
    \frac{d}{dt}\int e\,dx + \frac{1}{\nu} \int  \tilde{P}^2 \,dx = -\frac{1}{\nu} \int \tilde{P} \tilde{\mathcal{G}} \, dx.
\end{align}
 For all $\rho \geq 0$, in view of \eqref{norm} and $P(\rho)=\rho^\gamma$, we have
        \begin{align*}
            \int \rho(x)\left(\int_1^{\rho(x)}\frac{P^2(\tau)}{\tau^2}\mathrm{d}\tau\right)\mathrm{d}x
           &\leq  \int\rho P(\rho) \int_{1}^{\rho} \frac{P(\tau)}{\tau^2} \mathrm{d}\tau \mathrm{d}x \\
           &\leq P^*\int\bigl(e(\rho) + P(1)\rho\bigr)\mathrm{d}x- \int P(\rho)P(1)\mathrm{d}x\\
          & \leq P^*\left(\|e\|_{L^1}+P(1)\right)-\bar{P}P(1),
        \end{align*}
        which implies that
\begin{align}\label{3.26}
    \tilde{\Phi} \geq &\frac{1}{2} \int \left(
\mu |\nabla^{\perp}\cdot u|^2+\mu_2|\nabla w|^2+\xi|\nabla^{\perp}\cdot u-2w|^2 + \frac{1}{\nu} \left(
\tilde{\mathcal{G}}^2 + \tilde{P}^2
\right)
\right) dx\nonumber\\
&- \frac{1}{2\nu} \bigl(P^* \|e\|_{L^1} + P(1)(P^* - \bar{P})\bigr).
\end{align}
We now define the augmented energy functional
\begin{align*}
\Phi:=\tilde{\Phi}+\mathcal{E}+\|e\|_{L^1}+\frac{1}{2\nu}(P^*-P(1))P(1),
\end{align*}
which reads explicitly
\begin{align*}
\Phi
=&\frac{1}{2}\int \left(\rho|u|^2+\rho|w|^2+\mu|\nabla \mathcal{P}u|^2+\mu_2|\nabla w|^2+\xi|\nabla^{\perp}\cdot u-2w|^2+4e(\rho)\right)dx\\
&+\frac{1}{2\nu}\int \left(\tilde{\mathcal{G}}^2+\tilde{P}^2+\rho\int_{\rho}^1\frac{P^2(\tau)}{\tau^2}d\tau+(P^*-P(1))P(1)\right)dx.
\end{align*}
Then, using \eqref{3.26}, we arrive at
\begin{align}\label{3.27}
    \Phi\geq &\frac{1}{2}\left(\|\sqrt{\rho}u\|_{L^2}^2+\|\sqrt{\rho}w\|_{L^2}^2+\mu\|\nabla^{\perp}\cdot u\|_{L^2}^2+\mu_2\|\nabla w\|_{L^2}^2+\xi\|\nabla^{\perp}\cdot u-2w\|_{L^2}^2\right)\nonumber\\
    &+\frac{1}{2}\left(\frac{1}{\nu}(\|\tilde{\mathcal{G}}\|_{L^2}^2+\|\tilde{P}\|_{L^2}^2)+\left(4-\frac{P^*}{\nu}\right)\|e\|_{L^1}\right).
\end{align}

\bigbreak

We now combine the foregoing estimates to derive a closed differential inequality for the energy functional
$\Phi$. To control the interaction between the pressure and effective viscous flux on the right-hand side of  \eqref{3.24}, we first use Young's inequality to obtain
\[
\frac{1}{\nu} \int |\tilde{P}|\, |\tilde{\mathcal{G}}| \, dx \leq \frac{1}{\nu} \int  \tilde{\mathcal{G}}^2 \, dx + \frac{1}{4\nu} \int  \tilde{P}^2 \, dx.
\]
Next, the Poincar\'{e} inequality applied to $\tilde{\mathcal{G}}$ gives
\[
\frac{1}{\nu} \int  \left(
\frac{2}{\nu} \Lambda + 1
\right) \tilde{\mathcal{G}}^2 \, dx \leq \frac{8\rho^*}{\nu} \left(
\frac{2}{\nu} \|\Lambda\|_{L^{\infty}} + 1
\right) \frac{\|\nabla \mathcal{G}\|_{L^2}^2}{8\rho^*}.
\]
Recalling that $\|\tilde{\rho}\|_{L^2}^2 \leq(\rho^*)^2-1$ by the uniform bound \eqref{shangquejie}, we deduce from \eqref{3.23} that
\begin{align*}
    \frac{d}{dt} &\Phi
+ \frac{1}{4} \left\| \sqrt{\rho} \dot{u} \right\|_{L^2}^2
+ \frac{(\mu+\xi)^2}{8\rho^*} \left\|\nabla( \nabla^{\perp}\cdot u )\right\|_{L^2}^2
+ \frac{1}{8\rho^*}\left(1-\frac{8\rho^*}{\nu}\left(\frac{2}{\nu}\|\Lambda\|_{L^{\infty}}+1\right)\right) \left\| \nabla \mathcal{G} \right\|_{L^2}^2 \\
&+ \frac{1}{4} \left\| \sqrt{\rho} \dot{w} \right\|_{L^2}^2  +\frac{1}{4\nu}\|\tilde{P}\|_{L^2}^2+\frac{\mu_2^2}{12\rho^*}\|\Delta w\|_{L^2}^2+\xi\|\nabla^{\perp} \cdot u\|_{L^2}^2+\mu_2\|\nabla w\|_{L^2}^2 \\
&+  \int (\operatorname{div} u)^2\left(\frac{\Lambda}{2}+\nu\right)\, dx+\mu\|\nabla^{\perp} \cdot u\|_{L^2}^2 -C\frac{\rho^*\log(e+\rho^*)}{\nu^2}(P^{*})^2\|\nabla u\|_{L^2}^2\\
\leq&\frac{2\xi^2}{\rho^*}\|\nabla^{\perp} w\|_{L^2}^2+\frac{8\xi^2}{\rho^*}\|w\|_{L^2}^2+\frac{2\xi^2}{\rho^*}\|\nabla^{\perp}\cdot u\|_{L^2}^2\nonumber  \\
&+C\|\sqrt{\rho}u\|_{L^2}^2\|\nabla u\|_{L^2}^2\left(\frac{\rho^*\|\tilde{P}\|_{L^{\infty}}^2}{\nu^3}+\frac{(\rho^*)^2}{\nu^4}\|\tilde{\mathcal{G}}\|_{L^2}^2+\frac{(\rho^*)^2}{(\mu+\xi)^2}\|\nabla^{\perp} \cdot u\|_{L^2}^2+\frac{(\rho^*)^2\|\nabla w\|_{L^2}^2}{\mu_2^2}\right)\nonumber\\
&\quad\times\operatorname{log}\left(e+\|\tilde{\rho}\|_{L^2}+\frac{\|\rho\|_{L^2}\|\nabla u\|_{L^2}^2}{\|\sqrt{\rho}u\|_{L^2}^2}\right).
\end{align*}
Due to $\nu\geq \mu$, we have
\begin{align*}
    \mu\|\nabla u\|_{L^2}^2+(\lambda+\mu)\|\operatorname{div}u\|_{L^2}^2=\mu\|\nabla^{\perp} \cdot u\|_{L^2}^2+\nu \|\operatorname{div}u\|_{L^2}^2,
\end{align*}
which implies  
\begin{align*}
   \int  (\operatorname{div} u)^2\left(\frac{\Lambda}{2}+\nu\right)\, dx+\mu\|\nabla^{\perp} \cdot u\|_{L^2}^2\geq \frac{1}{2}\int  (\operatorname{div} u)^2\left(\Lambda+\nu\right)dx +\frac{\mu}{2}\|\nabla u\|_{L^2}^2  +\frac{\mu}{2}\|\nabla^{\perp} \cdot u\|_{L^2}^2.
\end{align*}
To simplify the logarithmic terms, note that for any $M\geq0$,
$$\log(e+\rho^*+\rho^*M)\leq\log (e+\rho^*)+\log(1+M)\leq \log(e+\rho^*)+M.$$
We now impose the following parameter constraints:
\begin{align}\label{3.29}
    \frac{16\rho^*}{\nu}\left(\frac{2}{\nu}\|\Lambda\|_{L^{\infty}}+1\right)\leq 1
\quad \text{and} \quad 4C\frac{\rho^*\log(e+\rho^*)(P^*)^2}{\mu\nu^2}\leq 1.
\end{align}
Using \eqref{3.29} and  the above bounds, we obtain
\begin{align}\label{3.30}
   & \frac{d}{dt}\Phi+\mathcal{D}
    \leq  C\rho^* \|\sqrt{\rho}u\|_{L^2}^2 \|\nabla u\|_{L^2}^2 \frac{\|\tilde{P}\|_{L^{\infty}}^2}{\nu^3}
\left(
\log(e + \rho^*) + \frac{\|\nabla u\|_{L^2}^2}{\|\sqrt{\rho}u\|_{L^2}^2}
\right)\nonumber\\
&+ C(\rho^*)^2 \|{\sqrt{\rho}u}\|_{L^2}^2 \|\nabla u\|_{L^2}^2
\left(
\frac{1}{(\mu+\xi)^2} \|\nabla^{\perp} \cdot u\|_{L^2}^2 + \frac{1}{\nu^4} \|\tilde{\mathcal{G}}\|_{L^2}^2+\frac{\|\nabla w\|_{L^2}^2}{\mu_2^2}
\right)\nonumber\\
&\cdot\left(
\log(e + \rho^*) + \log\left(1 + \frac{\|\nabla u\|_{L^2}^2}{\|\sqrt{\rho}u\|_{L^2}^2}\right)
\right)+\frac{2\xi^2}{\rho^*}\|\nabla^{\perp}w\|_{L^2}^2+\frac{8\xi^2}{\rho^*}\|w\|_{L^2}^2+\frac{2\xi^2}{\rho^*}\|\nabla^{\perp}\cdot u\|_{L^2}^2,
\end{align}
where the functional $\mathcal{D}$ is defined by
\begin{align*}
    \mathcal{D} := &
\frac{1}{4}\|\sqrt{\rho} \dot{u}\|_{L^2}^2 +\frac{1}{4}\|\sqrt{\rho} \dot{w}\|_{L^2}^2+ \frac{1}{12\rho^*}\mu_2^2\|\Delta w\|_{L^2}^2+\frac{(\mu+\xi)^2}{8\rho^*} \|\nabla(\nabla^{\perp}\cdot u)\|_{L^2}^2 + \frac{1}{16\rho^*} \|\nabla \mathcal{G}\|_{L^2}^2
 \nonumber \\
&+ \frac{1}{4\nu} \|\tilde{P}\|_{L^2}^2+\int \left( \frac{\nu + \Lambda}{2} \right) (\operatorname{div } u)^2dx + \frac{\mu+2\xi}{2} \|\nabla^{\perp}\cdot u\|_{L^2}^2+\frac{\mu}{4}\|\nabla u\|_{L^2}^2+\mu_2\|\nabla w\|_{L^2}^2.
\end{align*}
To close the estimates, we impose the following requirements on the bulk viscosity:
\begin{align}\label{3.31}
    \nu\geq\mu,\quad \nu^2\geq4C\mu^{-1}\rho^*\log (e+\rho^*)(P^*)^2,\quad \nu\geq 16\rho^*\left(2\nu^{-1}\|\Lambda\|_{L^{\infty}}+1\right)\quad  \text{and}\quad  \nu\geq \frac{P^*}{2}.
\end{align}
The last condition ensures the coefficient of $\|e\|_{L^1}$ in \eqref{3.27} is at least $1$, so that \eqref{3.4} holds and
\begin{align}\label{3.32}
\begin{aligned}
 &   \|\sqrt{\rho}u\|_{L^2}^2\leq 2\Phi,\quad\|\sqrt{\rho}w\|_{L^2}^2\leq 2\Phi,\quad \|\nabla^{\perp} \cdot u\|_{L^2}^2\leq \frac{2\Phi}{\mu},\\
 &\|\tilde{P}\|_{L^2}^2+ \|\tilde{\mathcal{G}}\|_{L^2}^2\leq 2\nu \Phi \quad \text{and} \quad \|\nabla w\|_{L^2}^2\leq \frac{2\Phi}{\mu_2}.
\end{aligned}
\end{align}
Substituting \eqref{3.32} into \eqref{3.30}, using
\eqref{energy} and the monotonicity of
\[
s\longmapsto s\log^{\frac12}\left(c_0+\frac{c_1}{s}\right),
\qquad c_0\geq1,\quad c_1\geq0,
\]
we obtain, 
\begin{align}\label{xiugai}
\frac{d}{dt} \Phi + \mathcal{D}
\leq& C \frac{(\rho^*)^2}{\mu^3} \mathcal{E}_0 \|\nabla u\|_{L^2}^2 \Phi \log\left(1 + \frac{\Phi}{\mu \mathcal{E}_0}\right)+\frac{2\xi^2}{\rho^*}\|\nabla^{\perp} w\|_{L^2}^2+\frac{8\xi^2}{\rho^*}\|w\|_{L^2}^2+\frac{2\xi^2}{\rho^*}\|\nabla^{\perp}\cdot u\|_{L^2}^2 \nonumber\\
&+ C \left( \frac{(\rho^*)^2}{\mu_3^3} \mathcal{E}_0 \log(e + \rho^*) + \frac{\rho^* \|\tilde{P}\|_{L^{\infty}}^2}{\mu \nu^3} + \frac{\rho^* \|\tilde{P}\|_{L^{\infty}}^2\log(e + \rho^*)}{ \nu^3} \right) \|\nabla u\|_{L^2}^2 \Phi,
\end{align}
where $\mu_3:=\min\{\mu,\mu_2\}$. From \eqref{3.31}, we  get
\[
\frac{\rho^* \|\tilde{P}\|_{L^{\infty}}^2}{\mu \nu^3} + \frac{\rho^* \|\tilde{P}\|_{L^{\infty}}^2\log(e + \rho^*)}{ \nu^3} \leq 1,
\]
which implies
\begin{align*}
\frac{d}{dt}\Phi+\mathcal D
\leq
C\frac{(\rho^*)^2}{\mu^3}\mathcal E_0
\|\nabla u\|_{L^2}^2\Phi
\log\left(1+\frac{\Phi}{\mu\mathcal E_0}\right)
+
C\left(
1+\frac{(\rho^*)^2}{\mu_3^3}\mathcal E_0
\log(e+\rho^*)
\right)
\|\nabla u\|_{L^2}^2\Phi+r(t),
\end{align*}
 where
\[
r(t):=\frac{2\xi^2}{\rho^*}
\left(
\|\nabla^\perp w\|_{L^2}^2
+4\|w\|_{L^2}^2
+\|\nabla^\perp\!\cdot u\|_{L^2}^2
\right).
\]
Since $\|\nabla^\perp w\|_{L^2}=\|\nabla w\|_{L^2}$ and
\[
4\|w\|_{L^2}^2
\leq
2\|\nabla^\perp\!\cdot u\|_{L^2}^2
+
2\|\nabla^\perp\!\cdot u-2w\|_{L^2}^2,
\]
the basic energy identity \eqref{energy} gives, for every $t\geq0$,
\begin{align}\label{xiuzheng1}
\int_0^t r(\tau)\,d\tau
\leq
\frac{2\xi^2}{\rho^*}
\left(
\frac1{\mu_2}+\frac3\mu+\frac2\xi
\right)\mathcal E_0.
\end{align}
Moreover, the periodic div--curl identity, \eqref{energy}, and
$\nu\geq\mu$ imply
\begin{align}\label{xiuzheng2}
\int_0^t\|\nabla u\|_{L^2}^2\,d\tau
&=
\int_0^t
\left(
\|\nabla^\perp\!\cdot u\|_{L^2}^2
+\|\operatorname{div}u\|_{L^2}^2
\right)d\tau
\nonumber\\
&\leq
\left(\frac1\mu+\frac1\nu\right)\mathcal E_0
\leq\frac{2\mathcal E_0}{\mu}.
\end{align}
Set
\[
\Psi(t):=\Phi(t)+\int_0^t\mathcal D(\tau)\,d\tau.
\]
Since $\Psi\geq\Phi$ and
\[
s\longmapsto
s\log\left(1+\frac{s}{\mu\mathcal E_0}\right)
\]
is increasing on $[0,\infty)$, we have
\begin{align*}
\Psi'(t)
\leq 
C\frac{(\rho^*)^2}{\mu^3}\mathcal E_0
\|\nabla u\|_{L^2}^2\Psi
\log\left(1+\frac{\Psi}{\mu\mathcal E_0}\right)
 +
C\left(
1+\frac{(\rho^*)^2}{\mu_3^3}\mathcal E_0
\log(e+\rho^*)
\right)
\|\nabla u\|_{L^2}^2\Psi+r(t).
\end{align*}
Consequently,
\begin{align*}
\frac{d}{dt}
\log\left(1+\frac{\Psi(t)}{\mu\mathcal E_0}\right)
\leq{}&
C\frac{(\rho^*)^2}{\mu^3}\mathcal E_0
\|\nabla u\|_{L^2}^2
\log\left(1+\frac{\Psi(t)}{\mu\mathcal E_0}\right)
\\
&+
C\left(
1+\frac{(\rho^*)^2}{\mu_3^3}\mathcal E_0
\log(e+\rho^*)
\right)\|\nabla u\|_{L^2}^2
+\frac{r(t)}{\mu\mathcal E_0}.
\end{align*}
Applying the usual Gronwall inequality and using \eqref{xiuzheng1} and \eqref{xiuzheng2}, we obtain
\begin{align}\label{3.5xiu}
\log\left(1+\frac{\Psi(t)}{\mu\mathcal E_0}\right)
\leq
\Bigg[ 
\log\left(1+\frac{\Phi_0}{\mu\mathcal E_0}\right)
 +
C\left(
1+\frac{(\rho^*)^2}{\mu_3^3}\mathcal E_0
\log(e+\rho^*)
\right)\frac{\mathcal E_0}{\mu}
 \nonumber\\
 +\frac{2\xi^2}{\rho^*\mu}
\left(
\frac1{\mu_2}+\frac3\mu+\frac2\xi
\right)
\Bigg]
\exp\left(
C\frac{(\rho^*)^2\mathcal E_0^2}{\mu^4}
\right).
\end{align}
Then exponentiating \eqref{3.5xiu} gives \eqref{3.5}.
In particular, the right-hand side of \eqref{3.5} is independent
of $t$ and $T$. This finishes the proof.
\end{proof}

\begin{rem}
    As a direct consequence of   \eqref{G} and   \eqref{3.4}, one obtains
    $$\nu\|\operatorname{div}u\|_{L^2}^2\leq 4\Phi.$$
\end{rem}

\bigbreak

\subsection{Uniform bounds on the density}
In this subsection, we establish a global-in-time $L^{\infty}$ estimate for the density under the assumption
that the bulk viscosity coefficient is sufficiently large.

\begin{prop}\label{prop2}
Let $(\rho, u, w)$ be a smooth solution of the Cauchy problem \eqref{0}--\eqref{eq:IC} on $[0,T]\times \mathbb{T}^2$ with initial data $(\rho_0, u_0, w_0)$ satisfying $\rho_0>0$ and  $\nu^{1/2}\|\operatorname{div}u_0\|_{L^2}\leq \digamma$. After enlarging the threshold of Proposition~\ref{prop1}, additionally let $\nu\geq \frac{16C_0^4}{\log^4 2}$. Then there exists $\nu_0>0$, depending on $\digamma$, $\gamma$, $\mu$, $\mu_2$, $\xi$, $\mathcal{E}_0$, $\|\rho_0\|_{L^{\infty}}$, $\|\nabla u_0\|_{L^2}$, and $\|\nabla w_0\|_{L^2}$, but independent of  time $T$, such that for all $\nu\geq \nu_0$,  one has

    \begin{align}\label{3.33}
        \sup_{t\in[0,T]} \|\rho(t)\|_{L^{\infty}} \leq 2e^{\frac{\gamma-1}{\gamma}\mathcal{E}_0 } \|\rho_0\|_{L^{\infty}}.
    \end{align}
\end{prop}

\begin{proof}
 By virtue of  \eqref{Gdingyi} and the positivity of  $\rho$, the material derivative of $\log \rho$ satisfies
$$\partial_t\log \rho+u\cdot \nabla \log \rho=-\frac{1}{\nu}(\tilde{\mathcal{G}}+\tilde{P}).$$
Taking the divergence of the momentum equation yields the elliptic equation for the fluctuation
$\tilde{\mathcal{G}}$:
$$\Delta\tilde{\mathcal{G}} = \partial_t   \operatorname{div} \left( \rho u \right)  + \operatorname{div} \left( \text{div} \left( \rho u \otimes u \right) \right).$$
Following the approach introduced in \cite{R.Danchin2,Desjardins1997},  we introduce the auxiliary function
$$\Upsilon \coloneqq \log \rho - \nu^{-1} (-\Delta)^{-1} \text{div} \left( \rho u \right),$$
which, using the summation convention over repeated indices, obeys the transport equation
\begin{equation}\label{3.34}
\partial_t \Upsilon + u \cdot \nabla \Upsilon + \frac{1}{\nu} \tilde{P} =  - \frac{1}{\nu} \left[ u^j, \left(-\Delta\right)^{-1} \partial_i \partial_j \right] \rho u^i.
\end{equation}
For the isentropic pressure law $P(\rho) = \rho^\gamma$, one readily verifies the identity $ \bar{P}-1 = (\gamma - 1) \| e \|_{L^1}$, which allows us to bound the final term by $(\gamma - 1) \mathcal{E}_0$. Applying the maximum principle for  transport equations, we obtain
\begin{align}\label{3.36}
\| \Upsilon^+(t) \|_{L^{\infty}}
&\leq e^{-\frac{\gamma}{\nu} t} \| \Upsilon^+(0) \|_{L^{\infty}}
+ \frac{1}{\nu} \int_0^t e^{-\frac{\gamma}{\nu} (t - \tau)} \left\| \left[ u^j, \left(-\Delta\right)^{-1} \partial_i \partial_j \right] \rho u^i(\tau) \right\|_{L^{\infty}} d\tau\nonumber \\
&\quad + \frac{\gamma}{\nu^2} \int_0^t e^{-\frac{\gamma}{\nu} (t - \tau)} \left\| \left(-\Delta\right)^{-1} \operatorname{div} \left( \rho u \right)(\tau) \right\|_{L^{\infty}} d\tau
 + \frac{\gamma - 1}{\gamma} \left( 1 - e^{-\frac{\gamma}{\nu} t} \right) \mathcal{E}_0.
\end{align}
The zero-mean property of $\rho u$,  together with the $L^4$-boundedness of Riesz transforms, implies
\begin{align*}
\left\| \left(-\Delta\right)^{-1} \text{div} \left( \rho u \right) \right\|_{L^{\infty}}
\lesssim
\left\| \left(-\Delta\right)^{-1} \nabla \operatorname{div} \left( \rho u \right) \right\|_{L^4}
\lesssim
\left\| \rho u \right\|_{L^4}.
\end{align*}
Then, we again use  \eqref{Desjardins1997} to obtain
\begin{equation*}
\left\| \left(-\Delta\right)^{-1} \operatorname{div} \left( \rho u \right) \right\|_{L^{\infty}}
\lesssim
\|\rho_0\|_{L^{\infty}}^{\frac{3}{4}}
\left\| \sqrt{\rho} u \right\|_{L^2}^{\frac{1}{2}}
\left\| \nabla u \right\|_{L^2}^{\frac{1}{2}}
\log^{\frac{1}{4}} \left(
e +\|\rho_0\|_{L^{\infty}} + \frac{\|\rho_0\|_{L^{\infty}} \left\| \nabla u \right\|_{L^2}^2}{\left\| \sqrt{\rho} u \right\|_{L^2}^2}
\right).
\end{equation*}
Consequently, by the energy balance \eqref{energy}, the monotonicity of   $s \mapsto s^{1/4} \log^{1/4}\left(c_0+ \frac{c_1}{s}\right)$ for $c_0 \geq 1$ and $c_1 \geq 0$, and the definition of $\Phi$ (under the assumption $\nu\geq \mu$), we obtain
\begin{equation}\label{3.38}
\left\| \left(-\Delta\right)^{-1} \operatorname{div} \left( \rho u \right) \right\|_{L^{\infty}}
\leq
C \|\rho_0\|_{L^{\infty}}^{\frac{3}{4}}
\mathcal{E}_0^{\frac{1}{4}}
\left\| \nabla u \right\|_{L^2}^{\frac{1}{2}}
\log^{\frac{1}{4}} \left(
e + \|\rho_0\|_{L^{\infty}} + \frac{\|\rho_0\|_{L^{\infty}} \Phi}{\mu \mathcal{E}_0}
\right).
\end{equation}
The commutator satisfies $[ u^j, \left(-\Delta\right)^{-1} \partial_i \partial_j ] \rho u^i = [ \tilde{u}^j, (-\Delta)^{-1} \partial_i \partial_j ] \rho u^i$ with $\tilde{u} = u - \bar{u}$, so we may apply the Coifman-Lions-Meyer-Semmes inequality  \cite{Coifman1993} to estimate the second term on the right-hand side of \eqref{3.36} as follows
\begin{equation}\label{3.39}
\left\| \left[ u^j, \left(-\Delta\right)^{-1} \partial_i \partial_j \right] \rho u^i \right\|_{L^{\infty}}
\lesssim
\left\| \left[ \tilde{u}^j, \left(-\Delta\right)^{-1} \partial_i \partial_j \right] \rho u^i \right\|_{W^{1,3}}
\lesssim
\left\| \nabla u \right\|_{L^{12}} \left\| \rho u \right\|_{L^4}.
\end{equation}
To bound $\nabla u$, we use the decomposition
\begin{align*}
\left\| \nabla u \right\|_{L^{12}}
\lesssim
\left\| \nabla \mathcal{P} u\right\|_{L^{12}} + \nu^{-1} \left(
\left\| \tilde{\mathcal{G}} \right\|_{L^{12}} + \left\| \tilde{P} \right\|_{L^{12}}
\right)
\lesssim
\left\| \nabla(\nabla^{\perp}\cdot u) \right\|_{L^2} + \nu^{-1} \left(
\left\| \nabla \mathcal{G} \right\|_{L^2} + \left\| \tilde{P} \right\|_{L^{\infty}}
\right).
\end{align*}
Proceeding as in the derivation of \eqref{3.38}, we arrive at
\begin{align*}
\begin{split}
\left\| [ u^j, \left(-\Delta\right)^{-1} \partial_i \partial_j ] \rho u^i \right\|_{L^{\infty}}
&\lesssim \|\rho_0\|_{L^{\infty}}^{\frac{3}{4}}
\left(
\left\| \nabla(\nabla^{\perp}\cdot u) \right\|_{L^2} + \nu^{-1} \left\| \nabla \mathcal{G} \right\|_{L^2} + \nu^{-1} \left\| \tilde{P} \right\|_{L^{\infty}}
\right) \\
&\quad \times \left\| \sqrt{\rho} u \right\|_{L^2}^{\frac{1}{2}}
\left\| \nabla u \right\|_{L^2}^{\frac{1}{2}}
\log^{\frac{1}{4}} \left(
e + \|\rho_0\|_{L^{\infty}} + \frac{\|\rho_0\|_{L^{\infty}}\left\| \nabla u \right\|_{L^2}^2}{\left\| \sqrt{\rho} u \right\|_{L^2}^2}
\right).
\end{split}
\end{align*}
As a result, using energy conservation \eqref{energy}  and the definitions of $\Phi$ and $\mathcal{D}$,
\begin{align}\label{3.40}
\left\| [ u^j, \left(-\Delta\right)^{-1} \partial_i \partial_j ] \rho u^i \right\|_{L^{\infty}}
\lesssim&
\left(
\|\rho_0\|_{L^{\infty}}^{\frac{5}{4}} (\mu+\xi)^{-1} \mathcal{D}^{\frac{1}{2}} + \|\rho_0\|_{L^{\infty}}^{\frac{3}{4}} \nu^{-1} \left\| \tilde{P} \right\|_{L^{\infty}}
\right)\nonumber\\
&\quad\times\mathcal{E}_0^{\frac{1}{4}} \left\| \nabla u \right\|_{L^2}^{\frac{1}{2}}
\log^{\frac{1}{4}} \left(
e +\|\rho_0\|_{L^{\infty}} + \frac{\|\rho_0\|_{L^{\infty}}\Phi}{\mu  \mathcal{E}_0}
\right).
\end{align}
Substituting  \eqref{3.38} and \eqref{3.40} into \eqref{3.36} and performing elementary simplifications yields
\begin{align}\label{3.41}
\left\| \Upsilon^+(t) \right\|_{L^{\infty}}
&\leq \left\| \Upsilon^+(0) \right\|_{L^{\infty}}+ \frac{\gamma - 1}{\gamma} \mathcal{E}_0 \nonumber\\
&\quad + C \frac{\|\rho_0\|_{L^{\infty}}^{\frac{3}{4}} \mathcal{E}_0^{\frac{1}{4}}}{\nu}
\int_0^t e^{-\frac{\gamma}{\nu} (t - \tau)}
\Bigg(
\left( \sqrt{\|\rho_0\|_{L^{\infty}}} (\mu+\xi)^{-1} \mathcal{D}^{\frac{1}{2}} + \nu^{-1} \left( \left\| \tilde{P} \right\|_{L^{\infty}} + \gamma \right) \right) \nonumber\\
&\qquad \times \left\| \nabla u \right\|_{L^2}^{\frac{1}{2}}
\log^{\frac{1}{4}} \left(
e + \|\rho_0\|_{L^{\infty}} + \frac{\|\rho_0\|_{L^{\infty}}\Phi}{\mu\mathcal{E}_0}
\right)
\Bigg)\, d\tau.
\end{align}

Now consider the maximal subinterval $[0,T_0]\subset[0,T]$ on which the density bound \eqref{3.33} holds. From \eqref{3.5}, there exist $\nu_0$ depending only on $\|\rho_0\|_{L^{\infty}}$, $\xi$,   $\mu$, $\mu_2$, and $\gamma$, and $C_0>0$ (depending also on $\digamma,\ \mathcal{E}_0,\ \xi, \ \|\nabla u_0\|_{L^2}$,  $\|\nabla w_0\|_{L^2}$  and $\|\rho_0\|_{L^{\infty}}$ )  such that for all $t\in[0,T_0]$ and $\nu\geq \nu_0$,
\begin{align}\label{3.42}
    \Phi(t)+\int_0^t\mathcal{D}(\tau)d \tau\leq C_0.
\end{align}
Since $\frac{1}{\mu}\geq\frac{1}{\mu+\xi}$,   inequality \eqref{3.41} reduces to
$$
\left\| \Upsilon^+ (t) \right\|_{L^{\infty}} \leq \left\| \Upsilon^+ (0) \right\|_{L^{\infty}} +
\frac{\gamma - 1}{\gamma} \mathcal{E}_0 +
\frac{C_0}{\nu} \int_0^t
e^{-\frac{\gamma}{\nu}(t - \tau)}
\left(
    \frac{\mathcal{D}^{\frac{1}{2}}(\tau)} {\mu} + \frac{1}{\nu}
\right)
\left\| \nabla u (\tau) \right\|_{L^2}^{\frac{1}{2}}
d\tau.
$$
By H\"{o}lder's inequality,  for any $t \in [0, T_0]$,
\begin{align*}
\int_0^t e^{-\frac{\gamma}{\nu}(t - \tau)} \mathcal{D}^{\frac{1}{2}} \left\| \nabla u \right\|_{L^2}^{\frac{1}{2}} d\tau
\leq C \left( \frac{\nu}{\gamma} \right)^{\frac{1}{4}}
\left( \int_0^t \mathcal{D}(\tau) d\tau \right)^{\frac{1}{2}}
\left( \int_0^t \left\| \nabla u(\tau) \right\|_{L^2}^2 d\tau \right)^{\frac{1}{4}},
\end{align*}
and similarly
\begin{align*}
\int_0^t e^{-\frac{\gamma}{\nu}(t - \tau)} \left\| \nabla u(\tau) \right\|_{L^2}^{\frac{1}{2}} d\tau
&\leq C \left( \frac{\nu}{\gamma} \right)^{\frac{3}{4}}
\left( \int_0^t \left\| \nabla u(\tau) \right\|_{L^2}^2 d\tau \right)^{\frac{1}{4}}.
\end{align*}
The integrals on the right are bounded by initial data through the basic energy balance \eqref{energy} and \eqref{3.42}. Consequently, adjusting $C_0$ if necessary, we obtain for $\nu \geq \nu_0$,
$$
\left\| \Upsilon^+ (t) \right\|_{L^{\infty}}\leq \left\| \Upsilon^+ (0) \right\|_{L^{\infty}} + C_0 \nu^{-\frac{1}{4}} + \frac{\gamma - 1}{\gamma} \mathcal{E}_0.
$$
From the definition of $\Upsilon^+$, \eqref{3.5} and \eqref{3.38},
$\log \rho \leq \Upsilon^++\nu^{-1}\|\Delta^{-1}\operatorname{div}(\rho u)\|_{L^{\infty}}\leq \Upsilon^++\nu^{-1}C_0$,
so that
\begin{align*} 
\log \|\rho(t)\|_{L^\infty} \leq \log \|\rho_0\|_{L^\infty} +C_0\nu^{-\frac{1}{4}}+\frac{\gamma-1}{\gamma}\mathcal{E}_0.
\end{align*}
Therefore, provided  $\nu$ is large enough to satisfy also
$C_0\nu^{-\frac{1}{4}}\leq \frac{1}{2}\log 2$,
which implies that
$\|\rho(t)\|_{L^\infty}\leq \sqrt{2}\|\rho_0\|_{L^\infty}\mathrm{e}^{\frac{\gamma-1}{\gamma}\mathcal{E}_0}<2\|\rho_0\|_{L^\infty}\mathrm{e}^{\frac{\gamma-1}{\gamma}\mathcal{E}_0}$.
Consequently, we have $T_0=T$, and \eqref{3.33} holds.
\end{proof}

\bigbreak

\subsection{Weighted estimates}

We now derive time-weighted estimates for the solution, which play a key role in the proof of Theorem \ref{result1}. Hereafter,  we assume  $\nu\geq\nu_0$, where $\nu_0$ is the threshold value determined in Propositions \ref{prop1} and \ref{prop2}. To simplify notation, we let $C_i$ ($i=0,1,2,\ldots$) denote positive constants that may depend on $T$,  $\gamma$, $\mu$, $\mu_2$, $\digamma$, $\xi$,  and the initial norms, but are independent of the bulk viscosity  $\nu$.
\begin{prop}\label{prop3}
 Let $T>0$  and suppose that $(\rho,u,w)$ is a smooth solution to the Cauchy problem  \eqref{0}--\eqref{eq:IC} posed on the two-dimensional torus $\mathbb{T}^2\times(0,T)$. Then, for every $\nu\geq\nu_0$, where $\nu_0$ is given in Propositions \ref{prop1} and  \ref{prop2}, there exists a positive constant $C(T)>0$, independent of $\nu$, such that
    \begin{align}\label{3.3}
\sup_{0\leq t\leq T}t&\left (\bigl\|\sqrt{\rho}\,\dot{u}(t)\bigr\|_{L^2}^2+\bigl\|\sqrt{\rho}\,\dot{w}(t)\bigr\|_{L^2}^2\right)
\nonumber\\
&+ \int_0^T t\left(\mu\bigl\|\nabla^\perp\cdot\dot{u}\bigr\|_{L^2}^2 + \nu\bigl\|\dot{Z}\bigr\|_{L^2}^2+\mu_2\|\nabla \dot{w}\|_{L^2}^2+\xi\|\nabla^{\perp}\cdot \dot{u}-2\dot{w}\|_{L^2}^2\right)dt\leq C(T),
\end{align}
where $Z := \operatorname{div}u$ and $\dot{Z} := \operatorname{div}u_t + u\cdot\nabla\operatorname{div}u$.

\end{prop}

\begin{proof}

The proof proceeds in several steps.
\bigbreak

\textbf{Step 1. Differential identities for $\|\sqrt{\rho}\dot{u}\|_{L^2}^2$ and $\|\sqrt{\rho}\dot{w}\|_{L^2}^2$.}

\bigbreak

We begin by rewriting the momentum equation $\eqref{0}_2$   in the equivalent form
\begin{equation}\label{3.4_1}
\rho \dot{u} + \nabla P(\rho) = (\mu + \xi)\nabla^{\perp}(\nabla^{\perp}\cdot u) + \nu\nabla \operatorname{div}u - 2\xi \nabla^{\perp} w.
\end{equation}
Applying   $\partial_{t} + \operatorname{div}(u\cdot)$ to both sides of \eqref{3.4_1}  and taking the  $L^{2}$ inner product with $\dot{u}$, we obtain
\begin{align}\label{3.5_1}
\frac{1}{2}\frac{d}{dt}\|\sqrt{\rho}\,\dot{u}\|_{L^{2}}^{2} = &(\mu + \xi)\int \dot{u}\cdot\Bigl[\nabla^{\perp}(\nabla^{\perp}\cdot u_{t}) + \operatorname{div}\bigl(u\nabla^{\perp}(\nabla^{\perp}\cdot u)\bigr)\Bigr]\,dx  \nonumber\\
&-\int\dot{u}\cdot\bigl(\nabla P_{t} + \operatorname{div}(u\nabla P)\bigr)\,dx
 -2\xi\int \dot{u}\cdot\Bigl(\partial_{t}(\nabla^{\perp}w) + \operatorname{div}(u\otimes\nabla^{\perp}w)\Bigr)\,dx\nonumber\\
 &+\nu\int \dot{u}\cdot\bigl(\nabla \operatorname{div} u_{t} + \operatorname{div}(u\nabla \operatorname{div}u)\bigr)\,dx.
\end{align}
We carry out an analogous procedure for the angular momentum equation  $\eqref{0}_3$. After rewriting the equation, applying the operator $\partial_{t} + \operatorname{div}(u\,\cdot)$, and testing with $\dot{w}$, we arrive at
\begin{align}\label{3.5_2}
   \frac{1}{2}\frac{d}{dt}\|\sqrt{\rho}\,\dot{w}\|_{L^{2}}^{2}+4\xi\|\dot{w}\|_{L^{2}}^{2}  = &-4\xi\int w\operatorname{div}u\,\dot{w}\,dx
    +\mu_2\int\bigl[\Delta w_t + \operatorname{div}(u\Delta w)\bigr]\dot{w}\,dx\nonumber\\
&+2\xi\int\bigl[\partial_t(\nabla^{\perp}\cdot u) + \operatorname{div}\left(u(\nabla^\perp\cdot u))\right)\bigr]\dot{w}\,dx,
\end{align}
where we have used 
$(\partial_{t} + \operatorname{div}(u\,\cdot))(4\xi w)=4\xi(\dot{w}+w\operatorname{div}u)$. 
Combining \eqref{3.5_1} with \eqref{3.5_2} yields
\begin{align}\label{3.5_3}
    &\frac{1}{2}\frac{d}{dt}\|\sqrt{\rho}\,\dot{u}\|_{L^{2}}^{2}+ \frac{1}{2}\frac{d}{dt}\|\sqrt{\rho}\,\dot{w}\|_{L^{2}}^{2}+4\xi\|\dot{w}\|_{L^{2}}^{2}  \nonumber\\
    = &(\mu + \xi)\int  \dot{u}\cdot\Bigl[\nabla^{\perp}(\nabla^{\perp}\cdot u_{t}) + \operatorname{div}\bigl(u\nabla^{\perp}(\nabla^{\perp}\cdot u)\bigr)\Bigr]\,dx +\nu\int  \dot{u}\cdot\bigl[\nabla \operatorname{div} u_{t} + \operatorname{div}(u\nabla \operatorname{div}u)\bigr]\,dx\nonumber\\
&-\int  \dot{u}\cdot\bigl[\nabla P_{t} + \operatorname{div}(u\nabla P)\bigr]\,dx-2\xi\int  \dot{u}\cdot\Bigl[\partial_{t}(\nabla^{\perp}w) + \operatorname{div}(u\otimes\nabla^{\perp}w)\Bigr]\,dx\nonumber\\
&+\mu_2\int \dot{w}\cdot\bigl[\Delta w_t + \operatorname{div}(u\Delta w)\bigr]\,dx+2\xi\int \dot{w}\cdot\bigl[\partial_t(\nabla^{\perp}\cdot u) + \operatorname{div}(u(\nabla^\perp\cdot u))\bigr]\,dx\nonumber\\
&-4\xi\int  w\operatorname{div}u\,\dot{w}\,dx:=Y_1+\cdots+Y_7.
\end{align}
We now estimate each term on the right-hand side of \eqref{3.5_3}. To this end, we repeatedly use the fact that, for all $\nu\geq \nu_0$ with $\nu_0$ given in Propositions \ref{prop1} and \ref{prop2},
\begin{align}\label{zhongyao1}
     \|\nabla u\|_{L^4(0,T;L^4(\mathbb{T}^2))}   \leq C_0.
\end{align}
This follows from the Helmholtz decomposition
$
u = \mathcal{P}u - \frac{1}{\nu} \nabla(-\Delta)^{-1}(\tilde{\mathcal{G}} + \tilde{P})$,
together with standard interpolation inequalities. Indeed, Proposition \ref{prop1} and the two-dimensional Sobolev embedding $H^1\hookrightarrow L^4$ imply
\begin{align*}
\|\nabla \mathcal{P} u\|_{L^4(0,T;L^4(\mathbb{T}^2))} \lesssim \|\nabla \mathcal{P} u\|_{L^\infty(0,T;L^2)}^{1/2} \|\nabla^2 \mathcal{P}u\|_{L^2(0,T;L^2)}^{1/2} \leq C_0.
\end{align*}
Moreover, the viscous effective flux and pressure fluctuations satisfy, respectively,
\begin{align*}
\|\tilde{\mathcal{G}}\|_{L^4(0,T;L^4(\mathbb{T}^2))} \lesssim \|\tilde{\mathcal{G}}\|_{L^\infty(0,T;L^2)}^{1/2} \|\nabla \mathcal{G}\|_{L^2(0,T;L^2)}^{1/2} \leq \nu^{1/4} C_0,
\end{align*}
and
\begin{align*}
\|\tilde{P}\|_{L^4(0,T;L^4(\mathbb{T}^2))} \leq \|\tilde{P}\|_{L^2(0,T;L^2)}^{1/2} \|\tilde{P}\|_{L^\infty(0,T;L^\infty)}^{1/2} \leq \nu^{1/4} C_0.
\end{align*}
 The terms $Y_1$-$Y_3$ are estimated similarly to the Navier-Stokes equations in \cite{liu2026}. For brevity, we omit redundant details. For $Y_1$, we have
\begin{align}\label{5.7}
Y_1 = -(\mu+\xi) \|\nabla^\perp \cdot \dot{u}\|_{L^2}^2 + (\mu+\xi) \int (\nabla^\perp \cdot u)\nabla \dot{u} : (\nabla^\perp \otimes u) \, dx,
\end{align}
where $\nabla^\perp \otimes u$ denotes the tensor product of $\nabla^\perp$ and $u$. Regarding $Y_2$ and $Y_3$, we obtain
\begin{align}\label{5.12}
Y_2 + Y_3
=& -\nu \int |\dot{Z}|^2 dx - \nu \int \dot{Z} \partial_i u \cdot \nabla u_i dx - \gamma \int P (\operatorname{div} \dot{u})Z\, dx \notag \\
&+ \int \partial_i \dot{u} \cdot \left( \mathcal{G}\nabla u_i \right) dx + \int Z \dot{u} \cdot \nabla \mathcal{G}\, dx   + \int \mathcal{G} \dot{u} \cdot \nabla Z\, dx.
\end{align}
Thus, inserting  \eqref{5.7} and \eqref{5.12} into \eqref{3.5_3}, we conclude that
\begin{align}\label{5.13}
&\frac{1}{2} \frac{d}{dt} \left\| \sqrt{\rho} \dot{u} \right\|_{L^2}^2+\frac{1}{2}\frac{d}{dt}\left\| \sqrt{\rho} \dot{w} \right\|_{L^2}^2  + ( \mu+\xi) \left\| \nabla^\perp \cdot \dot{u} \right\|_{L^2}^2 + \nu \left\| \dot{Z} \right\|_{L^2}^2 +4\xi\|\dot{w}\|_{L^2}^2 \notag \\
= &(\mu+\xi) \int (\nabla^\perp \cdot u)\nabla \dot{u} : (\nabla^\perp \otimes u) dx - \gamma \int P (\operatorname{div} \dot{u})Z dx \notag \\
&+ \int \partial_i \dot{u} \cdot \left( \mathcal{G}\nabla u_i \right) dx + \int Z \dot{u} \cdot \nabla \mathcal{G} dx
 + \int \mathcal{G} \dot{u} \cdot \nabla Z dx - \nu \int \dot{Z} \partial_i u \cdot \nabla u_i dx \notag \\
&-2\xi\int\dot{u}\cdot\left(\nabla^{\perp}w_t+\operatorname{div}(u\otimes \nabla^{\perp}w )\right)\,dx+\mu_2\int \dot{w}\cdot\bigl[\Delta w_t + \operatorname{div}(u\Delta w)\bigr]\,dx         \notag \\
&+2\xi\int \dot{w}\cdot\bigl[\partial_t(\nabla^{\perp}\cdot u) + \operatorname{div}(u(\nabla^\perp\cdot u))\bigr]\,dx-4\xi\int  w Z\,\dot{w}\,dx := L_1 + \dots + L_{10}.
\end{align}

\bigbreak

\textbf{Step 2. Estimates for  terms $L_1-L_6$.}
\bigbreak
Observing that
\begin{align}\label{5.11}
\operatorname{div} \dot{u} = (\operatorname{div} u)_t + \operatorname{div}\,(u \cdot \nabla u)= \dot{Z} + \partial_i u \cdot \nabla u_i,
\end{align}
 we start the estimation of the right-hand side of \eqref{5.13} by establishing the following key inequality:
\begin{align}\label{5.14}
    \left\| \dot{Z} \right\|_{L^2} - \left\| \nabla u \right\|_{L^4}^2 \leq \left\| \operatorname{div} \dot{u} \right\|_{L^2} \leq \left\| \dot{Z} \right\|_{L^2} + \left\| \nabla u \right\|_{L^4}^2.
\end{align}
Combining  \eqref{ding2.5} with \eqref{5.14},  we obtain
\begin{align}\label{5.15}
    |L_1| &\leq  C\|\nabla u\|_{L^4}^2 \|\nabla \dot{u}\|_{L^2}
\leq C\|\nabla u\|_{L^4}^2 \left( \|\nabla^\perp \cdot \dot{u}\|_{L^2} + \|\operatorname{div} \dot{u}\|_{L^2} \right)
\nonumber\\
&\leq  \frac{1}{32} \left( \mu\|\nabla^\perp \cdot \dot{u}\|_{L^2}^2 + \nu\|\dot{Z}\|_{L^2}^2 \right) + C\|\nabla u\|_{L^4}^4.
\end{align}
In light of \eqref{shangquejie}, \eqref{3.5} and \eqref{5.14}, it is straightforward to see that
\begin{align}\label{5.16}
    |L_2|\leq C\|\operatorname{div}\dot{u}\|_{L^2}\|Z\|_{L^2}\leq \frac{\nu}{32}\|\dot{Z}\|_{L^2}^2+C(1+\|\nabla u\|_{L^4}^4).
\end{align}
Based on elliptic regularity extracted from \eqref{3.4_1}, we first have
\begin{align} \label{elliptic_equation}
\Delta \mathcal{G} = \operatorname{div}(\rho \dot{u}) \quad \text{and} \quad \Delta(\nabla^\perp \cdot u) = \frac{1}{\mu+\xi}\nabla^\perp \cdot (\rho \dot{u})+\frac{2\xi}{\mu+\xi}\Delta w.
\end{align}
Consequently, using  \eqref{shangquejie} together with the Poincar\'{e} inequality $\|\mathcal{G}\|_{L^2} \leq C(1+\|\nabla \mathcal{G}\|_{L^2})$, we deduce
\begin{align}\label{5.18_1}
\|\mathcal{G}\|_{L^2} + \|\nabla \mathcal{G}\|_{L^2}  \leq C(1+\|\sqrt{\rho} \dot{u}\|_{L^2}),
\end{align}
as well as
\begin{align}\label{5.18_2}
\|\mathcal{G}\|_{L^2} + \|\nabla \mathcal{G}\|_{L^2} + \|\nabla(\nabla^\perp \cdot u)\|_{L^2} \leq C(1+\|\sqrt{\rho} \dot{u}\|_{L^2})+C\|\nabla w\|_{L^2}.
\end{align}
The terms $L_3$-$L_6$ are estimated analogously to the Navier-Stokes case in \cite{liu2026}, and we skip the full details for brevity. Combining
 \eqref{ding2.5}, \eqref{ding2.4}, \eqref{shangquejie}, \eqref{3.5}, \eqref{5.14} and \eqref{5.18_1}  yields 
\begin{align}\label{5.19}
\begin{cases}
    |L_3| \leq \frac{1}{32} \left( \mu\|\nabla^\perp \cdot \dot{u}\|_{L^2}^2 + \nu\|\dot{Z}\|_{L^2}^2 \right)
+ C \left( \|\nabla u\|_{L^4}^4 + \|\sqrt{\rho} \dot{u}\|_{L^2}^4 \right)\\
|L_4|  \leq \frac{1}{32} \left( \mu\|\nabla^\perp \cdot \dot{u}\|_{L^2}^2 + \nu\|\dot{Z}\|_{L^2}^2 \right)
+ C \left( 1 + \|\nabla u\|_{L^4}^4 + \|\sqrt{\rho} \dot{u}\|_{L^2}^4 \right),\\
|L_5|\leq \frac{1}{32} \left( \mu\|\nabla^\perp \cdot \dot{u}\|_{L^2}^2 + \nu\|\dot{Z}\|_{L^2}^2 \right)
+ C \left( 1 + \|\nabla u\|_{L^4}^4 + \|\sqrt{\rho} \dot{u}\|_{L^2}^4 \right),
\end{cases}
\end{align}
and
\begin{align}\label{5.31}
L_6  &\leq -\frac{d}{dt} \int \mathcal{G} \partial_i u \cdot \nabla u_i \, dx + \frac{1}{32} \left( \mu\|\nabla^\perp \cdot \dot{u}\|_{L^2}^2 + \nu\|\dot{Z}\|_{L^2}^2 \right)
  + C \left( 1 + \|\nabla u\|_{L^4}^4 + \|\sqrt{\rho} \dot{u}\|_{L^2}^4 \right).
\end{align}

\bigbreak

\textbf{Step 3. Estimates for the terms $L_7-L_{10}$.}
\bigbreak

It remains to estimate $L_7$, $L_8$, $L_9$ and $L_{10}$. 
Note that
\[
\partial_t(\nabla^\perp w)
+\operatorname{div}(u\otimes\nabla^\perp w)
=
\nabla^\perp\dot w
+(\nabla^\perp w\cdot\nabla)u,
\]
and
$\partial_t(\nabla^\perp\!\cdot u)
+\operatorname{div}\bigl(u(\nabla^\perp\!\cdot u)\bigr)
=
\nabla^\perp\!\cdot\dot u$,
it follows that
\begin{align}\label{7+9}
L_7+L_9
=&-2\xi\int\nabla^\perp\dot w\cdot\dot u\,dx
-2\xi\int
\bigl((\nabla^\perp w\cdot\nabla)u\bigr)\cdot\dot u\,dx
+2\xi\int
(\nabla^\perp\!\cdot\dot u)\dot w\,dx
\nonumber\\
=&4\xi\int
(\nabla^\perp\!\cdot\dot u)\dot w\,dx
-2\xi\int
\bigl((\nabla^\perp w\cdot\nabla)u\bigr)\cdot\dot u\,dx.
\end{align}
For the last integral in \eqref{7+9}, using \eqref{ding2.4} and \eqref{5.14} yields
\begin{align}\label{7+91}
\left|
2\xi\int
\bigl((\nabla^\perp w\cdot\nabla)u\bigr)\cdot\dot u\,dx
\right|
\leq&
C\|\nabla w\|_{L^2}
\|\nabla u\|_{L^4}
\|\dot u\|_{L^4}\nonumber\\
   \leq &C\|\nabla u\|_{L^4}(\|\sqrt{\rho}\dot{u}\|_{L^2}+\|\nabla \dot{u}\|_{L^2})\nonumber\\
   \leq & C\|\nabla u\|_{L^4} (\|\sqrt{\rho}\dot{u}\|_{L^2}+\|\operatorname{div}\dot{u}\|_{L^2}+\|\nabla^{\perp}\cdot\dot{u}\|_{L^2})\nonumber\\
   \leq &C\|\nabla u\|_{L^4}(\|\sqrt{\rho}\dot{u}\|_{L^2}+\|\dot{Z}\|_{L^2}+\|\nabla u\|_{L^4}^2+\|\nabla^{\perp}\cdot \dot{u}\|_{L^2})\nonumber\\
   \leq & \frac{1}{32}(\mu\|\nabla^{\perp}\cdot \dot{u}\|_{L^2}^2+\nu\|\dot{Z}\|_{L^2}^2)+C(1+\|\nabla u\|_{L^4}^4+\|\sqrt{\rho}\dot{u}\|_{L^2}^4).
\end{align}
Then combining \eqref{7+9} and \eqref{7+91}, one has
\begin{align}\label{7+92}
    L_7+L_9\leq 4\xi\int \nabla^{\perp}\cdot\dot{u}\cdot\dot{w}\, dx+\frac{1}{32}(\mu\|\nabla^{\perp}\cdot \dot{u}\|_{L^2}^2+\nu\|\dot{Z}\|_{L^2}^2)+C(1+\|\nabla u\|_{L^4}^4+\|\sqrt{\rho}\dot{u}\|_{L^2}^4).
\end{align}
We now estimate $L_8$. First, note the identity
\begin{align}\label{N89}
(\partial_t+\operatorname{div}(u\cdot))(\Delta w)=\Delta \dot{w}-2(\nabla u):(\nabla^2 w) -(\Delta u)\cdot (\nabla w)+(\Delta w)\operatorname{div}u.
\end{align}
It follows that
\begin{align}\label{N90}
    L_8&=\mu_2\int \Delta \dot{w}\cdot\dot{w}\,dx-\mu_2\int(2\nabla u:\nabla^2w+\Delta u\cdot\nabla w-\Delta w\operatorname{div}u )\dot{w}\,dx\nonumber\\
    &=-\mu_2\|\nabla \dot{w}\|_{L^2}^2-\mu_2\int(2\nabla u:\nabla^2w+\Delta u\cdot\nabla w-\Delta w\operatorname{div}u )\dot{w}\,dx.
\end{align}
To bound  the  integral in \eqref{N90}, we rewrite   \eqref{0}$_3$ to obtain
\begin{align}\label{N8w}
    \|\Delta w\|_{L^2}\leq C(\mu_2,\xi,\rho^*)(\|\sqrt{\rho}\dot{w}\|_{L^2} +\| w\|_{L^2}+\|\nabla^{\perp}\cdot u\|_{L^2}).
\end{align}
For the term containing $\Delta u$, we first integrate by parts:
\begin{align*}
\int \Delta u_j\,\partial_j w\,\dot w\,dx
=-\int \partial_i u_j\,\partial_{ij}w\,\dot w\,dx-\int \partial_i u_j\,\partial_jw\,
\partial_i\dot w\,dx.
\end{align*}
Then, using Lemma~\ref{lemma2.2}, Proposition~\ref{prop1},  $\|\nabla w\|_{L^4}
\leq C\|\nabla w\|_{L^2}^{1/2}
\|\Delta w\|_{L^2}^{1/2}$,
together with \eqref{N8w} and Young's inequality, we arrive at
\begin{align*} 
    &-\mu_2\int(2\nabla u:\nabla^2w+\Delta u\cdot\nabla w-\Delta w\operatorname{div}u )\dot{w}\,dx\nonumber\\
    \leq& C\left(\|\Delta w\|_{L^2}\|\dot{w}\|_{L^4}\|\nabla u\|_{L^4}+\|\nabla u\|_{L^4}
\|\nabla w\|_{L^4}
\|\nabla\dot w\|_{L^2}\right) \nonumber\\
    \leq&  \frac{\mu_2}{4}\|\nabla \dot{w}\|_{L^2}^2+C(1+\|\Delta w\|_{L^2}^4+\|\nabla u\|_{L^4}^4)\nonumber\\
    \leq& \frac{\mu_2}{4}\|\nabla \dot{w}\|_{L^2}^2+C(1+\|\sqrt{\rho}\dot{w}\|^4_{L^2}+\|w\|^4_{L^2}+\|\nabla^{\perp}\cdot u\|^4_{L^2}+\|\nabla u\|_{L^4}^4).
\end{align*}
Thus,
\begin{align}\label{N9N}
    L_8\leq-\frac{3\mu_2}{4}\|\nabla \dot{w}\|_{L^2}^2+C\left(1+\|\sqrt{\rho}\dot{w}\|^4_{L^2}+\|w\|^4_{L^2}+\|\nabla^{\perp}\cdot u\|^4_{L^2}+\|\nabla u\|_{L^4}^4\right).
\end{align}
 Analogously,
\begin{align}\label{N10}
    |L_{10}|&\leq C \|\dot{w}\|_{L^4}\|\nabla u\|_{L^4}\|w\|_{L^2}\nonumber\\
    &\leq \frac{\mu_2}{4}\|\nabla \dot{w}\|_{L^2}^2+C\left(\|\sqrt{\rho}\dot{w}\|^4_{L^2}+\|w\|^4_{L^2}+\|\nabla u\|_{L^4}^4\right).
\end{align}

\bigbreak

\textbf{Step 4. Proof of \eqref{3.3}.}
\bigbreak

Now, inserting \eqref{5.15}, \eqref{5.16}, \eqref{5.19}, \eqref{5.31}, \eqref{7+92}, \eqref{N9N}  and \eqref{N10} into \eqref{5.13}, we derive
\begin{align}\label{5.32}
 &\frac{d}{dt} \left\| \sqrt{\rho} \dot{u} \right\|_{L^2}^2+\frac{d}{dt}\left\| \sqrt{\rho} \dot{w} \right\|_{L^2}^2  + \mu \left\| \nabla^\perp \cdot \dot{u} \right\|_{L^2}^2 + \nu \left\| \dot{Z} \right\|_{L^2}^2 +\mu_2\|\nabla \dot{w}\|_{L^2}^2+2\xi\|\nabla^{\perp}\cdot \dot{u}-2\dot{w}\|_{L^2}^2 \notag \\
\leq &-2\frac{d}{dt}\int \mathcal{G}\partial_i u\cdot\nabla u_i\, dx +C\left(1+\|\nabla u\|_{L^4}^4+\|\sqrt{\rho} \dot{u}\|_{L^2}^4+\|\sqrt{\rho}\dot{w}\|_{L^2}^4\right) .
\end{align}
It remains to control the time-derivative term on the right-hand side of \eqref{5.32}. Since
$\partial_i u \cdot \nabla u_i = Z^2 + 2\nabla^\perp u_2 \cdot \nabla u_1$,
we decompose the integral as
\begin{align}\label{5.33}
2 \int \mathcal{G} \partial_i u \cdot \nabla u_i \, dx &= 2 \int \mathcal{G} Z^2 \, dx + 4 \int \mathcal{G} \nabla^\perp u_2 \cdot \nabla u_1 \, dx.
\end{align}
Recalling that $\operatorname{div} \, u = \nu^{-1}(\mathcal{G}+P)$ and using  $\nu^{-1}\|\mathcal{G}\|_{L^2} \leq C$   from \eqref{shangquejie} and \eqref{3.5}, together with Lemma \ref{GNS} and \eqref{5.18_1}, we obtain
\begin{align}\label{5.34}
\left| 2 \int \mathcal{G} Z^2 \, dx \right|  &= \left| \frac{2}{\nu} \int \mathcal{G} (\mathcal{G} + P ) Z dx \right|\nonumber\\
&\leq C \left( \nu^{-1}\|\mathcal{G}\|_{L^4}^2 + \nu^{-1}\|\mathcal{G}\|_{L^2} \right) \|\mathcal{G}\|_{L^2} \notag \\
&\leq C \left( \nu^{-1}\|\mathcal{G}\|_{L^2}\|\nabla \mathcal{G}\|_{L^2} + \nu^{-1}\|\mathcal{G}\|_{L^2} \right) \|\mathcal{G}\|_{L^2} \leq C + \frac{1}{4}\left\| \sqrt{\rho} \dot{u} \right\|^2_{L^2}.
\end{align}
Then, since $\operatorname{BMO}(\mathbb{T}^2)$ and $\mathcal{H}^1(\mathbb{T}^2)$ are dual spaces, together with $\nabla^\perp \cdot (\nabla u_1) = \operatorname{div}(\nabla^\perp u_2) = 0$, we have via  \eqref{BMO}, \eqref{3.5} and \eqref{5.18_1}
\begin{align}\label{5.35}
\left| 4 \int \mathcal{G} \nabla^\perp u_2 \cdot \nabla u_1 \, dx \right| &\leq C\|\mathcal{G}\|_{\operatorname{BMO}}\|\nabla^\perp u_2 \cdot \nabla u_1\|_{\mathcal{H}^1}
 \leq C\|\nabla \mathcal{G}\|_{L^2}\|\nabla u\|_{L^2}^2
 \leq C\left\| \sqrt{\rho} \dot{u} \right\|_{L^2}.
\end{align}
Consequently, substituting \eqref{5.34} and \eqref{5.35} into \eqref{5.33}, we conclude
\begin{align}\label{5.36}
\left| 2 \int \mathcal{G} \partial_i u \cdot \nabla u_i \, dx \right| &\leq C + \frac{1}{4}\left\|\sqrt{\rho} \dot{u} \right\|^2_{L^2}.
\end{align}
Finally, multiplying \eqref{5.32} by $t$, integrating over $(0, T)$, and applying \eqref{3.5}, \eqref{5.36} together with the Cauchy-Schwarz and Gronwall inequalities, we complete the proof of the desired estimate \eqref{3.3}.
\end{proof}

The following corollary follows immediately from Propositions \ref{prop1}--\ref{prop3}, and is fundamental to establishing Theorem \ref{result1}.
\begin{cor}\label{cor3.1}
 Suppose the conditions of Proposition \ref{prop3} are satisfied. Then, for any $2< q<\infty$ and $1<p<\frac{q}{q-1}$, there exists a constant $C(p,q, T)>0$ such that
\begin{align}\label{5.104}
\int_{0}^{T}
\biggl(
&\|\rho \dot{u}\|_{L^q}^{p} + \|\nabla \mathcal{G}\|_{L^q}^{p}
+ \|\nabla(\nabla^\perp \cdot u)\|_{L^q}^{p}+\|\rho \dot{w}\|_{L^q}^{p}
\biggr) dt\nonumber\\
&+\int_{0}^{T}
\biggl(
\|\mathcal{G}\|_{L^\infty}^{p} + \|\nabla^\perp \cdot u\|_{L^\infty}^{p} + \nu^{p}\|\mathrm{div}\,u\|_{L^\infty}^{p}+\|w\|^{p}_{L^{\infty}}
\biggr) dt \leq C(p,q, T).
\end{align}

\end{cor}
\begin{proof}

First, for any $2 < q < \infty$ and  $1<p<\frac{q}{q-1}$, we can choose a constant $C=C(p,q,T)$ such that 
\begin{align}\label{5.108}
&\int_{0}^{T} \biggl(\|\rho \dot{u}\|_{L^q}^{p} + \|\nabla \mathcal{G}\|_{L^q}^{p}
+ \|\nabla(\nabla^\perp \cdot u)\|_{L^q}^{p} \nonumber \\
& \quad + \|\mathcal{G}\|_{L^\infty}^{p} + \|\nabla^\perp \cdot u\|_{L^\infty}^{p} +  \nu^{p}\|\mathrm{div}\,u\|_{L^\infty}^{p}
\biggr)   \mathrm{d}t \leq C(q, p,T).
\end{align}
Since the estimate of \eqref{5.108}  follows from the argument in \cite{liu2026}, we omit the details here for brevity.

In what follows, we focus on the estimates of $\|\rho \dot{w}\|_{L^q}^{p}$  and $\|w\|^{p}_{L^{\infty}}$. Since $2 < q < \infty$, a combination of  the interpolation inequality and \eqref{shangquejie} ensures the existence of $0 < \theta < 1$ and $q_{*}> q$, where $q_{*}$ depends on $q$ and $\theta$, such that
\begin{align}
\|\rho \dot{w}\|_{L^q}
\leq C\|\rho \dot{w}\|_{L^2}^{\theta}\|\rho\dot{w}\|_{L^{q_*}}^{1-\theta} \leq C\|\sqrt{\rho} \dot{w}\|_{L^2}^{\theta}\|\dot{w}\|_{L^{q_*}}^{1-\theta},
\end{align}
which, combined with  Lemma \ref{lemma2.2}, implies
\begin{align*}
\|\rho\dot w\|_{L^q}
 \le C\|\sqrt\rho\dot w\|_{L^2}^{\theta}
\Bigl(\|\sqrt\rho\dot w\|_{L^2}
+\|\nabla\dot w\|_{L^2}\Bigr)^{1-\theta}.
\end{align*}
Set
\[
A(t)=\|\sqrt\rho\dot w(t)\|_2,
\qquad B(t)=\|\nabla\dot w(t)\|_2.
\]
Using Proposition~\ref{prop1} and \eqref{3.3},
\[
A\in L^2(0,T),\qquad
\sup_{0<t\le T}tA(t)^2\le C_T,
\qquad
\int_0^T tB(t)^2\,dt\le C_T.
\]
If $\theta>\frac{2(p-1)}{p}$,
then H\"{o}lder's inequality yields
\[
\int_0^T
\bigl(A^\theta B^{1-\theta}\bigr)^{p}\,dt
\le C_T.
\]
Consequently,
\begin{equation}\label{5.108xin}
\int_0^T\|\rho\dot w\|_{L^q}^{p}\,dt
\le C(p,q,T).
\end{equation}
 Next, invoking the Sobolev embedding $W^{1,q}(\mathbb{T}^2) \hookrightarrow L^\infty(\mathbb{T}^2)$  for  $q>2$ together with \eqref{ding2.2}, \eqref{3.5}, \eqref{5.18_2} and \eqref{N8w},  one deduces that for any $2 < q < \infty$ and $1<p<\frac{q}{q-1}$,
\begin{align}\label{5.110}
\int_{0}^{T}
\|w\|^{p}_{L^{\infty}}dt
\leq C \int_{0}^{T}
(
\|w\|^{p}_{L^2}+\|\nabla w\|^{p}_{L^q}) dt
\leq C.
\end{align}
Combining   \eqref{5.108}, \eqref{5.108xin} and \eqref{5.110}, we  complete   the proof of \eqref{5.104}.
\end{proof}

\section{Proof of Theorems \ref{result1} and \ref{result2} }

With the complete set of a priori estimates established in Section 3 at hand, we are now in a position to prove Theorems \ref{result1}--\ref{result2}. At the beginning of this section, we introduce an approximate sequence $(\rho^{\delta},u^{\delta}, w^{\delta})$ corresponding to the regularized initial data $(\rho^{\delta}_0,u^{\delta}_0, w^{\delta}_0)$.

The original initial data satisfy the following regularity assumptions:
\begin{align*}
    \rho_0\in L^{\infty}(\mathbb{T}^2)\quad \text{and} \quad u_0,w_0\in H^1(\mathbb{T}^2).
\end{align*}
To ensure strict positivity of the density while preserving the unit total mass, we proceed as follows. For $\delta \in (0, 1)$, define
\begin{align*}
\tilde{\rho}_0^\delta = \max\{\rho_0, \delta\} \quad \text{and then set} \quad \check{\rho}_0^\delta = \min\{\zeta^\delta, \tilde{\rho}_0^\delta\},
\end{align*}
where the normalizing constant $\zeta^\delta \geq 1$ is chosen such that
\begin{align*}
\int_{\mathbb{T}^2} \check{\rho}_0^\delta \, dx = 1.
\end{align*}
It is straightforward to verify that $\zeta^\delta \to   \|\rho_0\|_{L^\infty}$ as $\delta \to 0$. This construction immediately yields uniform bounds and pointwise convergence:
\begin{align*}
\delta \leq \check{\rho}_0^\delta \leq  \|\rho_0\|_{L^\infty}, \quad   \quad \check{\rho}_0^\delta \to \rho_0 \ \text{a.e.}
\end{align*}
The bounded approximate density $\check{\rho}_0^\delta$ and the original velocity fields are subsequently mollified via convolution with a family of positive smooth mollifiers  $(\pi^\delta)_{\delta>0}$:
\begin{align*}
\rho_0^\delta = \pi^\delta * \check{\rho}_0^\delta \quad  \quad u_0^\delta = \pi^\delta * u_0 \quad \text{and} \quad w_0^\delta = \pi^\delta * w_0.
\end{align*}
For any finite $T > 0$, the approximate solutions constructed above satisfy the a priori estimates \eqref{3.5}, \eqref{3.33}, and \eqref{3.3}, with a constant $C$ independent of $\delta$. Passing to the limit $\delta \to 0$ and employing classical compactness arguments from \cite{R.Danchin5,R.Danchin2}, we conclude that the Cauchy problem \eqref{0}--\eqref{eq:IC} admits a global  solution $(\rho, u, w)$ satisfying \eqref{1.11}.

We now establish the uniqueness of the solution $(\rho, u, w)$ constructed in Theorem \ref{result1} under the linear pressure law $P= \rho$, i.e.  for the case $\gamma=1$.  

Suppose  $(\rho,u,w)$ and $(\bar{\rho},\bar{u},\bar{w})$ are two   solutions to the Cauchy problem \eqref{0}--\eqref{eq:IC} on $[0,T_0]\times\mathbb{T}^2$ for any $T_0 \in (0,T]$ satisfying \eqref{1.11} with the same initial data. Combining Propositions \ref{prop1}--\ref{prop3} and Corollary \ref{cor3.1}, we obtain 
$u,w,\bar{u},\bar{w}\in L^{\infty}(0,T_0;H^1(\mathbb{T}^2))$,  $\sqrt{t}\nabla \dot{u}, \sqrt{t}\nabla \dot{w}, \sqrt{t}\nabla \dot{\bar{u}}, \sqrt{t}\nabla \dot{\bar{w}}\in L^2(0,T_0;L^2(\mathbb{T}^2))$, $\sqrt{t\rho}\dot{u}, \sqrt{t\rho} \dot{w}, \sqrt{t\bar{\rho}} \dot{\bar{u}}$ and $\sqrt{t\bar{\rho}} \dot{\bar{w}}$   in $L^{\infty}(0,T_0;L^2(\mathbb{T}^2))$,
\begin{align*}
 \nabla \bar{u},\nabla \bar{w} \in L^2(0,T_0;L^3(\mathbb{T}^2))\quad \text{and}\quad \int_0^{T_0}(1+|\log t|)\|\nabla \bar{u}\|_{\mathrm{BMO}}\,d t< \infty.
\end{align*}

 We first define the difference quantities
$\delta\rho:=\rho-\bar{\rho}, \quad \delta u:=u-\bar{u}\quad \text{and} \quad \delta w:=w-\bar{w}$. These satisfy the error system
\begin{align}\label{5.2}
\begin{cases}
\delta \rho_t + \operatorname{div}\left( \delta \rho \bar{u} + \rho \delta u \right) = 0, \\
\rho \dot{\delta u}- (\mu+\xi) \Delta \delta u - (\lambda + \mu-\xi) \nabla \operatorname{div} \delta u + \nabla \delta \rho+2\xi \nabla^{\perp}(\delta w)= -(\delta \rho \dot{\bar{u}} + \rho \delta u \cdot \nabla \bar{u}),\\
\rho\dot{\delta w}-\mu_2\Delta \delta w+4\xi\delta w-2\xi \nabla^{\perp}\cdot(\delta u)=-(\delta \rho \dot{\bar{w}} + \rho \delta u \cdot \nabla \bar{w}),
\end{cases}
\end{align}
where $\dot{\delta u}:=\delta u_t+u\cdot \nabla \delta u$ and analogously for $\dot{\delta w}$.

We aim to show that $\delta \rho \equiv 0$, $\delta u \equiv 0$ and $\delta w \equiv 0$ by establishing estimates for $\delta \rho$ in $L^\infty(0, T; \dot{H}^{-1})$, and for $\sqrt{\rho} \delta u$ and $\sqrt{\rho} \delta w$ in $L^2(0,T \times \mathbb{T}^2)$. To this end, introduce $\kappa \coloneqq -(-\Delta)^{-1} \delta \rho$, which  is well-defined since $\int  \delta \rho \, dx = 0$. Observe the norm equivalence
\begin{align}\label{7.3}
\|\nabla \kappa\|_{L^2} = \|\delta \rho\|_{\dot{H}^{-1}} \simeq \|\delta \rho\|_{H^{-1}}.
\end{align}
 Multiplying  the first equation of \eqref{5.2} by $\kappa$, we derive
\begin{align*}
\frac{1}{2} \frac{d}{dt} \|\nabla \kappa\|_{L^2}^2 \leq \left| \int  \left( \bar{u} \cdot \nabla \kappa \, \delta \rho + \rho \, \delta u \cdot \nabla \kappa \right) dx \right|.
\end{align*}
We bound the two integrals separately. For the second term, H\"{o}lder's inequality gives
\begin{align*}
\left| \int  \rho \delta u \cdot \nabla \kappa \, dx \right| \leq \sqrt{\rho^*} \, \|\sqrt{\rho} \delta u\|_{L^2} \|\nabla \kappa\|_{L^2}.
\end{align*}
For the first term, using Einstein summation and integration by parts,
\begin{align*}
\int  \bar{u} \cdot \nabla \kappa \, \delta \rho \, dx &= \int  \bar{u}^j \partial_j \kappa \Delta \kappa \, dx
 = -\int  \partial_k \bar{u}^j \, \partial_j \kappa \partial_k \kappa \, dx + \frac{1}{2} \int \operatorname{div} \bar{u} \, |\nabla \kappa|^2 \, dx.
\end{align*}
By  the div-curl structure and   $\mathcal{H}^1$-$\text{BMO}$ duality,
\begin{align}\label{10.4}
\left| \int  \bar{u} \cdot \nabla \kappa \, \delta \rho \, dx \right| \leq C \|\nabla \bar{u}\|_{\text{BMO}} \|\nabla \kappa \otimes \nabla \kappa\|_{\mathcal{H}^1}.
\end{align}
Applying the logarithmic interpolation inequality for Hardy spaces (cf. \cite{Mucha2010}, Theorem D)
\begin{align}\label{7.4}   \|f\|_{\mathcal{H}^1}\leq C\|f\|_{L^1}(|\log\|f\|_{L^1}|+\log \left(e+\|\nabla \kappa\|_{L^{\infty}}^2)\right),
\end{align}
we infer from \eqref{10.4} that
\begin{align}\label{7.5}
    \left|\int \bar{u}\cdot\nabla \kappa\delta \rho\, dx\right|\leq C\|\nabla \bar{u}\|_{\mathrm{BMO}}\|\nabla \kappa\|_{L^2}^2\left(|\log \|\nabla \kappa\|_{L^2}^2|+\log(e+\|\nabla \kappa\|^2_{L^{\infty}})\right).
\end{align}
From the uniform bound $\rho\leq\rho^*$, we have
$\|\nabla \kappa(t)\|_{L^{\infty}}\leq C\rho^*$.
Substituting this into \eqref{7.5},  there exists a constant $C$ depending only on $\rho^*$ such that
\begin{align}\label{7.6}
\frac{1}{2} \frac{d}{dt} \|\nabla \kappa\|_{L^2}^2 \leq C \left( \|\sqrt{\rho} \delta u\|_{L^2} + \|\nabla \bar{u}\|_{\text{BMO}} \|\nabla \kappa\|_{L^2} \left(1 + \left| \log \|\nabla \kappa\|_{L^2} \right| \right) \right) \|\nabla \kappa\|_{L^2}.
\end{align}
Combined with  $\|\nabla \kappa(0)\|_{L^2} = 0$, integrating \eqref{7.6} over time yields for all $t \in [0, T_0]$,
\begin{align*}
\|\nabla \kappa(t)\|_{L^2} \leq C  \int_0^t \|\sqrt{\rho} \delta u\|_{L^2} \, d\tau + C\int_0^t \|\nabla \bar{u}\|_{\text{BMO}} \|\nabla \kappa\|_{L^2} \left(1 + \left| \log \|\nabla \kappa\|_{L^2} \right| \right) d\tau.
\end{align*}
Define $\mathcal{J}(t):=\sup_{\tau\leq t}\tau^{-\frac{1}{2}}\|\delta \rho(\tau)\|_{\dot{H}^{-1}(\mathbb{T}^2)}$. Using
\eqref{7.3} and Cauchy-Schwarz inequality for any $T\in[0,T_0]$,
\begin{equation}\label{7.7}
\mathcal{J}(T) \leq C \left(
\int_0^T \|\nabla \bar{u}\|_{\text{BMO}} \, \mathcal{J}(\tau) \bigl(1 + |\log \tau| + |\log \mathcal{J}(\tau)|\bigr) d\tau +
\bigl\| \sqrt{\rho} \delta u \bigr\|_{L^2(0,T \times \mathbb{T}^2)}
\right).
\end{equation}

Unlike the Navier--Stokes equations \cite{R.Danchin2}, the micropolar fluid model needs to control both the velocity difference $\delta u$ and the micro-rotation velocity difference $\delta w$. To this end, we introduce the following backward parabolic system:
\begin{equation}\label{7.8}
\begin{cases}
\rho \psi_t + \rho u \cdot \nabla \psi + (\mu + \xi) \Delta \psi + (\mu + \lambda - \xi) \nabla \operatorname{div} \psi - 2\xi \nabla^{\perp} \chi = -\rho \delta u, \\
\rho \chi_t + \rho u \cdot \nabla \chi - 4\xi \chi +\mu_2 \Delta \chi + 2\xi \nabla^\perp \cdot \psi = -\rho \delta w, \\
\psi \big|_{t=T} = 0,\quad \chi \big|_{t=T} = 0,
\end{cases}
\end{equation}
where $ \nabla^{\perp} \chi:= (-\partial_2 \chi, \partial_1 \chi)$ and $
\nabla^\perp \cdot \psi := \partial_1 \psi_2 - \partial_2 \psi_1$.  Testing the first equation in
\eqref{7.8} against $\psi$ and the second against $\chi$, summing the resulting identities yields
\begin{align}\label{5.123}
    \sup_{t\in(0,T)}\int (\rho|\psi|^2 +\rho|\chi|^2)(t) dx + &\int_{0}^{T}\int \left( \mu|\nabla^{\perp} \cdot \psi|^2 + \nu(\text{div}\ \psi)^2+\mu_2|\nabla \chi|^2+\xi|\nabla^{\perp}\cdot\psi-2\chi|^2 \right) dx dt\nonumber \\
    &\leq \int_{0}^{T}\left(\| \sqrt{\rho}\delta u\|_{L^2}^2+\|\sqrt{\rho}\delta w\|_{L^2}^2\right)\, dt.
\end{align}
Rewriting  \eqref{7.8}, we obtain
\begin{align*}
    (\mu+\xi)\Delta \psi +(\mu+\lambda-\xi)\nabla\operatorname{div}\psi =-\rho \delta u-\rho \dot{\psi}+2\xi \nabla^{\perp}  \chi
\end{align*}
and
   \begin{align*}
    &\mu_2\Delta \chi=-\rho \delta w-\rho \dot{\chi}+4\xi \chi-2\xi \nabla^{\perp}\cdot\psi.
\end{align*}
Then, applying standard elliptic estimates for the Hodge decomposition, we derive
\begin{align}\label{phi5}
(\mu+\xi)^2\|\nabla(\nabla^{\perp}\cdot\psi)\|_{L^2}^2+\nu^2\|\nabla \operatorname{div} \psi\|_{L^2}^2&= \|\rho \delta u+\rho \dot{\psi}-2\xi \nabla^{\perp}\chi\|_{L^2}^2\nonumber\\
  &  \leq 2\rho^*\|\sqrt{\rho}(\dot{\psi}+\delta u)\|_{L^2}^2+8\xi^2\|\nabla^{\perp} \chi\|_{L^2}^2,
\end{align}
and
\begin{align}\label{chi5}
    \mu_2^2\|\Delta \chi\|_{L^2}^2\leq3(\rho^*\|\sqrt{\rho}(\delta w+\dot{\chi})\|_{L^2}^2+16\xi^2\|\chi\|_{L^2}^2+4\xi^2\|\nabla^{\perp}\cdot \psi\|_{L^2}^2).
\end{align}
Next,  testing \eqref{7.8}$_{1}$ by $\psi_t$ and \eqref{7.8}$_{2}$ by $\chi_t$, and combining  with \eqref{phi5}-\eqref{chi5}, we obtain
\begin{align}\label{7.9}
&\sup_{0\leq t\leq T}
\left(
\mu\|\nabla ^{\perp}\cdot\psi\|_{L^2}^2+\mu_2\|\nabla\chi\|_{L^2}^2 + \nu\|\mathrm{div}\, \psi\|_{L^2}^2
+\|\sqrt{\rho}\psi\|^2_{L^2}+\|\sqrt{\rho}\chi\|^2_{L^2}+\xi\|2\chi-\nabla^{\perp}\cdot\psi\|_{L^2}^2\right)\nonumber\\
&\quad+
\int_0^T
\left(
\mu^2\|\nabla(\nabla^{\perp} \cdot\psi)\|_{L^2}^2 + \nu^2\|\nabla\mathrm{div}\, \psi\|_{L^2}^2 + \|\sqrt{\rho}\psi_t\|_{L^2}^2+  \|\sqrt{\rho}\chi_t\|_{L^2}^2+\mu_2^2\|\Delta \chi\|_{L^2}^2
\right)
dt\nonumber\\
\leq& C_T
\int_0^T
\int
\left(\rho|\delta u|^2 + \rho|u \cdot \nabla \psi|^2+\rho|\delta w|^2 + \rho|u \cdot \nabla \chi|^2\right) dx dt,
\end{align}
with $C_T$ depending only on the solution norms.
By direct calculations, we have
\begin{align*}
    \int \rho|u\cdot \nabla \psi|^2\,dx&\leq \sqrt{\rho^*}\|\rho^{\frac{1}{4}}u\|_{L^4}^2\|\nabla \psi\|_{L^4}^2\leq \sqrt{\rho^*}\|\rho^{\frac{1}{4}}u\|_{L^4}^2\|\nabla \psi\|_{L^2}\|\nabla^2 \psi\|_{L^2}\\
    &\leq \epsilon \|\nabla^2 \psi\|_{L^2}^2+C\|\rho^{\frac{1}{4}}u\|_{L^4}^4\|\nabla \psi\|_{L^2}^2,
\end{align*}
and analogously for $\chi$,
    \begin{align*}
    \int \rho|u\cdot \nabla \chi|^2\,dx\leq \epsilon \|\nabla^2 \chi\|_{L^2}^2+C\|\rho^{\frac{1}{4}}u\|_{L^4}^4\|\nabla \chi\|_{L^2}^2.
\end{align*}
Consequently, by using \eqref{5.123}, \eqref{7.9} and the regularity properties of $u,w$, we have  
\begin{align}\label{7.10}
&\sup_{0\leq t\leq T}
\left(
\mu\|\nabla ^{\perp}\cdot\psi\|_{L^2}^2 + \mu_2\|\nabla\chi\|_{L^2}^2+\nu\|\mathrm{div}\, \psi\|_{L^2}^2
+\|\sqrt{\rho}\psi\|^2_{L^2}+\|\sqrt{\rho}\chi\|^2_{L^2}\right)\nonumber\\
&\quad+
\int_0^T
\left(
\mu^2\|\nabla(\nabla^{\perp}\cdot \psi)\|_{L^2}^2 + \nu^2\|\nabla\mathrm{div}\, \psi\|_{L^2}^2 + \|\sqrt{\rho}\psi_t\|_{L^2}^2+  \|\sqrt{\rho}\chi_t\|_{L^2}^2+\mu_2^2\|\Delta \chi\|_{L^2}^2
\right)
dt\nonumber\\
&\quad+\int_{0}^{T} \left( \mu\|\nabla^{\perp} \cdot \psi\|_{L^2}^2 + \nu\|\text{div}\ \psi\|_{L^2}^2+\mu_2\|\nabla \chi\|_{L^2}^2+\xi\|\nabla^{\perp}\cdot\psi-2\chi\|_{L^2}^2 \right) dt\nonumber\\
\leq& C_T
\int_0^T
\int
\left(\rho|\delta u|^2 +\rho|\delta w|^2\right) dx dt.
\end{align}
Next, we  estimate all terms on the right-hand side of \eqref{7.10}. Testing  \eqref{5.2}$_2$ by $\psi$ and \eqref{5.2}$_3$ by $\chi$, and then summing the resulting equalities, we arrive at
\begin{align}\label{7.11}
 &\int_0^T \int(
\rho |\delta u|^2+\rho |\delta w|^2 )dx dt -
\int_0^T \int
\delta\rho \, \mathrm{div}\, \psi \, dx dt  \nonumber\\
\leq&
\int_0^T \int
\left(
\delta\rho  \dot{\bar{u}} \cdot   \psi + \rho (\delta u \cdot \nabla \bar{u}) \cdot \psi
\right) dx dt
 +\int_0^T \int
\left(
\chi\delta\rho \, \dot{\bar{w}}  + \chi\rho (\delta u \cdot \nabla \bar{w})   
\right) dx dt.
\end{align}
We may treat each term in \eqref{7.11} using the same argument employed in \cite{R.Danchin2}. Consequently, we  deduce 
\begin{align}\label{5.1235}
&\left\| \sqrt{\rho}\delta u \right\|_{L^2(0,T; L^2)}^2+\left\| \sqrt{\rho}\delta w \right\|_{L^2(0,T; L^2)}^2
\leq C T^{1/3}
\left(
 \left\| \sqrt{\rho}\delta u \right\|_{L^2(0,T; L^2)}^2
+ \left\| \sqrt{\rho}\delta w \right\|_{L^2(0,T; L^2)}^2
\right)\nonumber\\
&\qquad+C T^{1/3}
\left(
\left\| t^{-1/2}\delta\rho \right\|_{L^\infty(0,T; \dot{H}^{-1})}\left(
\left\| \sqrt{\rho}\delta u \right\|_{L^2(0,T; L^2)}+\left\| \sqrt{\rho}\delta w \right\|_{L^2(0,T; L^2)}\right)
\right).
\end{align}
By definition of $\mathcal{J}(T)$, for sufficiently small $T$, \eqref{5.1235}  reduces to
\begin{align}\label{7.17}
    \left\| \sqrt{\rho}\delta u \right\|_{L^2(0,T; L^2)} +&\left\| \sqrt{\rho}\delta w \right\|_{L^2(0,T; L^2)} \leq CT^{1/3}\mathcal{J}(T).
\end{align}
Inserting  \eqref{7.17} into \eqref{7.7} yields the logarithmic integral inequality
\begin{align*}
    \mathcal{J}(t) \leq C_T \int_0^t
\left(1 + \|\nabla \bar{u}(\tau)\|_{\mathrm{BMO}} \left(1 + |\log \tau| + |\log \mathcal{J}(\tau)|\right)\right) \mathcal{J}(\tau) d\tau, \qquad\forall\,t\in[0,T].
\end{align*}
 By the classical Osgood's lemma (\cite{Bahouri2011}, Lemma 3.4), we conclude that $\mathcal{J}\equiv 0$ on $[0,T]$,
 whence $\delta \rho \equiv 0$. In view of \eqref{7.17}, this further implies $\sqrt{\rho}\delta u \equiv0$ and $\sqrt{\rho}\delta w \equiv0$ on $[0,T]$.

Now, since $\sqrt{\rho}\delta u =0$, $\sqrt{\rho}\delta w =0$ and $\delta \rho=0$, the  second and third equations in \eqref{5.2} simplify  to
\begin{align*}
\begin{cases}
\rho \dot{\delta u}- (\mu+\xi) \Delta \delta u - (\lambda + \mu-\xi) \nabla \operatorname{div} \delta u +2\xi \nabla^{\perp}(\delta w)= 0,\\
\rho \dot{\delta w}+4\xi\delta w-\mu_2\Delta \delta w-2\xi \nabla^{\perp}\cdot(\delta u)=0.
\end{cases}
\end{align*}
The direct energy estimate then implies
\begin{align*}
  &  \frac{1}{2}\|\sqrt{\rho}\delta u\|^2_{L^{\infty}(0,T;L^2)}+\frac{1}{2}\|\sqrt{\rho}\delta w\|^2_{L^{\infty}(0,T;L^2)}\\
   &\quad +\int_0^T(\mu\|\nabla^{\perp}\cdot \delta u\|_{L^2}^2+\nu\|\operatorname{div}\delta u\|_{L^2}^2+\mu_2\|\nabla \delta w\|_{L^2}^2+\xi\|2\delta w-\nabla^{\perp}\cdot \delta u\|_{L^2}^2)\, dt= 0,
\end{align*}
which in turn yields $\delta u=0$ and $\delta w=0$ on $[0,T]$. The uniqueness of the solution thus follows.

\bigbreak

The proof of Theorem \ref{result2} follows a compactness argument similar to Theorem \ref{result1}. For this reason, we only sketch the key steps.

\bigbreak

\noindent \textbf{Proof of Theorem \ref{result2}.}
We now outline the convergence of solutions to  \eqref{0} toward those of the incompressible micropolar system \eqref{Incom} as $\nu \to \infty$, which completes the proof of  Theorem~\ref{result2}. The argument relies on the compactness method developed in the proof of Theorem 2.3 in \cite{R.Danchin2}.
 The whole analysis is based on the uniform estimates derived from Proposition \ref{prop1}.  In particular,
    Proposition \ref{prop1} implies that the sequence $(\nabla \mathcal{G}^\nu)$ is bounded in $L^2(\mathbb{R}_+ \times \mathbb{T}^2)$, while the velocity and micro-rotation fields $(u^\nu,w^{\nu})$ are uniformly bounded in $L^\infty(\mathbb{R}_+; H^1(\mathbb{T}^2))$. Meanwhile, Proposition \ref{prop2} ensures that the density sequence  $(\rho^\nu)$ is bounded in
$L^\infty(\mathbb{R}_+ \times \mathbb{T}^2)$. Consequently, there exists a subsequence $(\rho^{\nu_n}, u^{\nu_n},w^{\nu_n})$ of $(\rho^\nu,u^\nu,w^{\nu})$  and limit functions $(\eta,v,w)$ such that
\[
\rho^{\nu_n} \overset{*}{\rightharpoonup} \eta \quad \text{in} \ L^\infty(\mathbb{R}_+ \times \mathbb{T}^2),\quad  u^{\nu_n} \overset{*}{\rightharpoonup} v \quad \text{in} \ L^\infty(\mathbb{R}_+; H^1(\mathbb{T}^2))
\quad \text{and} \quad
w^{\nu_n} \overset{*}{\rightharpoonup} w \quad \text{in} \ L^\infty(\mathbb{R}_+; H^1(\mathbb{T}^2)).
\]
To upgrade weak convergence to strong convergence, we use the time-weighted estimates from Proposition \ref{prop3}. More precisely, for $\nu\geq \nu_0$ and any $T>0$,
 \begin{align*}
\sup_{0\leq t\leq T}t &\left (\bigl\|\sqrt{\rho^{\nu}}\,\dot{u}^{\nu}(t)\bigr\|_{L^2}^2+\bigl\|\sqrt{\rho^{\nu}}\,\dot{w}^{\nu}(t)\bigr\|_{L^2}^2\right)
\\
&+ \int_0^T t\left(\mu\bigl\|\nabla^\perp\cdot\dot{u}^{\nu}\bigr\|_{L^2}^2 + \nu\bigl\|\dot{Z}^{\nu}\bigr\|_{L^2}^2+\mu_2\|\nabla \dot{w}^{\nu}\|_{L^2}^2\right)dt\leq C(T),
\end{align*}
where $Z := \operatorname{div}u$ and $\dot{Z} := \operatorname{div}u_t + u\cdot\nabla\operatorname{div}u$. This implies that  $(\sqrt{t}u^{\nu}_t)$ and $(\sqrt{t}w^{\nu}_t)$ are bounded in $L^2(0,T;L^2)$. 
Fix $0<\delta<T$. By Proposition~\ref{prop3} and Lemma~\ref{lemma2.2},
\[
\dot u^\nu,\ \dot w^\nu
\quad\text{are bounded in }L^2(\delta,T;L^2(\mathbb T^2)),
\]
and the uniform $H^1$ bounds imply
\[
u_t^\nu=\dot u^\nu-u^\nu\cdot\nabla u^\nu,
\qquad
w_t^\nu=\dot w^\nu-u^\nu\cdot\nabla w^\nu
\]
are bounded in $L^{4/3}(\delta,T;H^{-1}(\mathbb T^2))$.
Since $u^\nu,w^\nu$ are bounded in
$L^2(\delta,T;H^1(\mathbb T^2))$, the Aubin--Lions lemma yields,
after extraction,
\[
(u^{\nu_n},w^{\nu_n})\to(v,w)
\quad\text{in }L^2(\delta,T;L^2(\mathbb T^2)).
\]
Moreover, the uniform $L^\infty(0,T;L^2)$ bounds give
\[
\int_0^\delta\|(u^{\nu_n}-v,w^{\nu_n}-w)\|_{L^2}^2\,dt\le C\delta.
\]
Letting $\delta\downarrow0$ and using a diagonal argument, we obtain
\[
(u^{\nu_n},w^{\nu_n})\to(v,w)
\quad\text{in }L^2_{\rm loc}(\mathbb R_+;L^2(\mathbb T^2)).
\]
Since
\[
\overline{\mathcal G^\nu}=-\overline{P(\rho^\nu)},
\qquad
|\overline{\mathcal G^\nu}|\le P^*,
\]
and Proposition~\ref{prop1} gives
$\nabla\mathcal G^\nu$ bounded in $L^2(0,T;L^2)$,
Poincar\'e inequality yields
\[
\|\mathcal G^\nu\|_{L^2(0,T;H^1)}\le C_T.
\]
Hence, after extraction,
\[
\nabla\mathcal G^{\nu_n}\rightharpoonup-\nabla\Pi
\quad\text{in }L^2(0,T;L^2(\mathbb T^2)).
\]
Combining all the above convergence results, we pass to the limit in the system and obtain
\begin{align*}
&\partial_t\eta +\operatorname{div}(\eta v)=0,\\
&\partial_t(\eta v)+\operatorname{div}(\eta v\otimes v)-(\mu+\xi)\Delta v+\nabla \Pi = -2\xi\nabla^{\perp}w, \quad\text{and}\\
&\partial_t(\eta w)  + \operatorname{div}(\eta v   w) + 4\xi w = \mu_2\Delta w + 2\xi\nabla^{\perp}\cdot v.
\end{align*}
Since $\operatorname{div}u^{\nu_n}\to0$ in
$L^2(0,T;L^2(\mathbb T^2))$, the limit velocity satisfies
$\operatorname{div}v=0$. Passing to the limit in the weak
formulation and using
$\nabla\mathcal G^{\nu_n}\rightharpoonup-\nabla\Pi$, we conclude
that $(\eta,v,w,\Pi)$ solves \eqref{Incom} with the prescribed
initial data. By the regularization techniques of \cite{R.Danchin5}, every sequence $\nu_n\to\infty$ contains a
subsequence for which all the convergences stated in Theorem~\ref{result2} hold. This proves Theorem~\ref{result2}. \qed

 \section*{Acknowledgement}
The authors declare that no funds, grants, or other support were received
during the preparation of this manuscript.

\section*{Declarations}

\noindent {\textbf{ Conflict of interest.} } The authors declare that they have no conflict of interest.

\noindent {\textbf{Data availability statement.} }No datasets were generated or analyzed during the current study.

\noindent\textbf{Clinical trial number.}
Not applicable. This article does not involve a clinical trial.

\vspace{3ex}

\vspace{3ex}

(Zijian Du)\par\nopagebreak
\noindent\textsc{School of Mathematics and Statistics, Northwestern Polytechnical University, Xi'an 710129, P.R. China}

Email address: \texttt{duzijian@mail.nwpu.edu.cn}

\vspace{3ex}

(Sen Yang)\par\nopagebreak
\noindent\textsc{School of Mathematics and Statistics, Northwestern Polytechnical University, Xi'an 710129, P.R. China}

Email address: \texttt{senyang@mail.nwpu.edu.cn}

 \vspace{3ex}

(Xinglong You)\par\nopagebreak
\noindent\textsc{School of Mathematics and Statistics, Northwestern Polytechnical University, Xi'an 710129, P.R. China}

 Email address: \texttt{xinglongyou@mail.nwpu.edu.cn}

\end{document}